\documentclass[12pt,letterpaper]{article}

\usepackage{amsmath}
\usepackage{amssymb}
\usepackage{amsthm}
\usepackage{graphicx}
\usepackage{mathtools} 
\usepackage[all]{xy}

\newcommand{\Z}{\mathbb{Z}}
\newcommand{\Q}{\mathbb{Q}}
\newcommand{\R}{\mathbb{R}}
\newcommand{\C}{\mathbb{C}}
\renewcommand{\H}{\mathbb{H}}
\newcommand{\GL}{\mathrm{GL}}
\newcommand{\SL}{\mathrm{SL}}
\newcommand{\bs}{\backslash}
\DeclareMathOperator{\Tr}{\mathrm{Trace}}
\DeclareMathOperator{\vcd}{\mathrm{vcd}}

\newtheorem*{corollary}{Corollary}
\newtheorem*{lemma}{Lemma}
\newtheorem*{proposition}{Proposition}
\newtheorem{theorem}{Theorem}

\theoremstyle{definition}
\newtheorem{definition}{Definition}

\theoremstyle{remark}
\newtheorem*{remark}{Remark}

\allowdisplaybreaks

\begin{document}

\title{Computing Hecke Operators for Arithmetic Subgroups of General
  Linear Groups}
\author{
  Mark McConnell\footnote{Dept.\ of Mathematics, Princeton University.  \texttt{markwm@princeton.edu}} \and
  Robert MacPherson\footnote{School of Mathematics, Institute for Advanced Study.  \texttt{rdm@ias.edu}}
}

\maketitle

\begin{abstract}
  We present an algorithm to compute the Hecke operators on the
  equivariant cohomology of an arithmetic subgroup~$\Gamma$ of the
  general linear group $\GL_n$.  This includes $\GL_n$ over a number
  field or a finite-dimensional division algebra.  As coefficients, we
  may use any finite-dimensional local coefficient system.  Unlike
  earlier methods, the algorithm works for the cohomology $H^i$ in all
  degrees~$i$.  It starts from the well-rounded retract
  $\widetilde{W}$, a $\Gamma$-invariant cell complex which computes
  the cohomology~\cite{Ash84}.  It extends~$\widetilde{W}$ to a new
  \emph{well-tempered complex} $\widetilde{W}^+$ of one higher real
  dimension, using a real parameter called the \emph{temperament}.
  The algorithm has been coded up for~$\SL_n(\Z)$ for $n=2,3,4$; we
  present some results for congruence subgroups of $\SL_3(\Z)$.
\end{abstract}

\section{Introduction}

\subsection{The $2\times 2$ Case}

We begin by describing our Hecke operator algorithm in the simplest
case.  Let~$E$ be the space of real symmetric matrices $Z =
\left[\begin{smallmatrix} a&b\\b&c \end{smallmatrix}\right]$.  The
group $G = \GL_2(\R)$ acts on~$E$ by $Z \mapsto g Z g^t$.  The action
preserves the cone of positive definite matrices~$X \subset E$.  We
mod out in~$X$ by the positive scalar multiples of the identity, the
\emph{homotheties}.  Then the action (when $\det g > 0$) is equivalent
to the well-known action of $\SL_2(\R)$ on the upper half-plane
$\mathfrak{H}\subset\C$ by linear fractional transformations.  $X$ is
called the \emph{Klein model} of~$\mathfrak{H}$.  The action
on~$\mathfrak{H}$ has the advantage of being complex-analytic, but the
Klein model has a different advantage: $X$ has a \emph{linear}
structure inherited from~$E$.

A matrix $Z \in X$ defines a positive definite quadratic form
on~$\R^2$, which we use to measure the lengths of vectors in the
standard lattice $L_0 = \Z^2$.  Divide~$Z$ by a homothety so the
shortest non-zero vector in $L_0$ has length~$1$.  We say~$Z$ is
\emph{well rounded} if two linearly independent vectors in~$\Z^2$ have
length~$1$.  Figure~\ref{fig:pencil} shows the $Z=1$ level curves of
five forms, which are well rounded because $(1,0)$, $(0,1)$ have
$Z=1$.  The figure suggests there is a continuous path, a pencil of
well-rounded forms, connecting the first and last pictures.

\begin{figure}
  \begin{center}
    \includegraphics[scale=0.15]{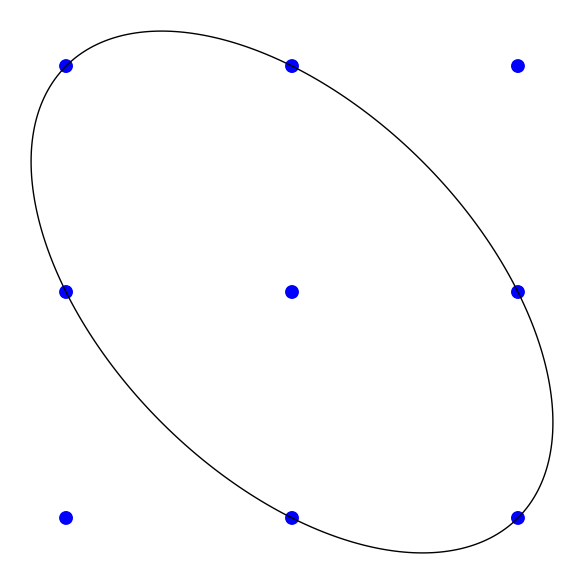} $\,\,\,\,$
    \includegraphics[scale=0.15]{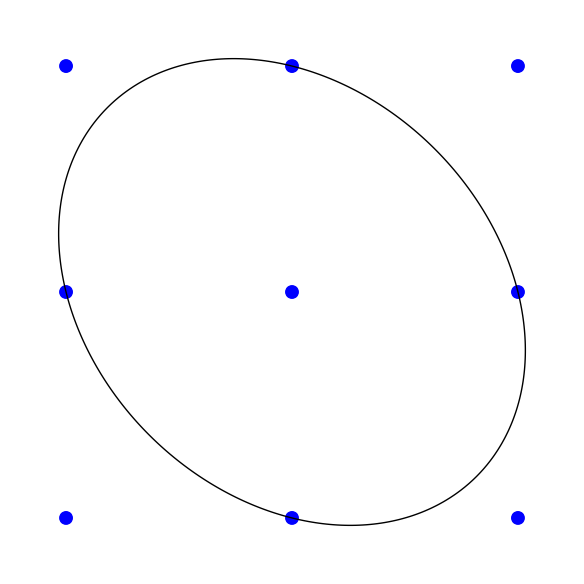} $\,\,\,\,$
    \includegraphics[scale=0.15]{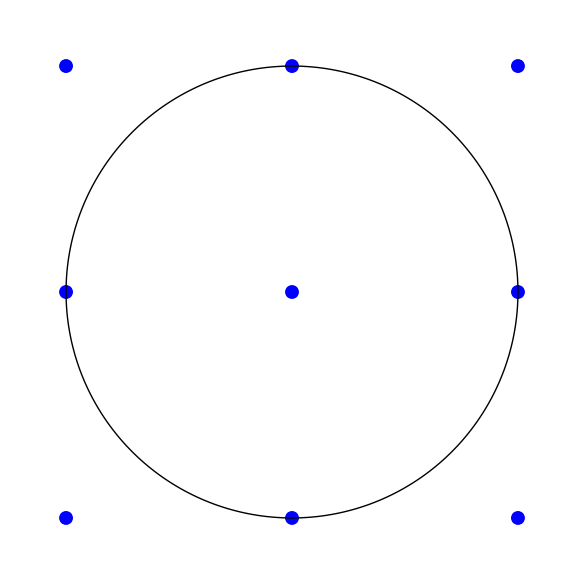} $\,\,\,\,$
    \includegraphics[scale=0.15]{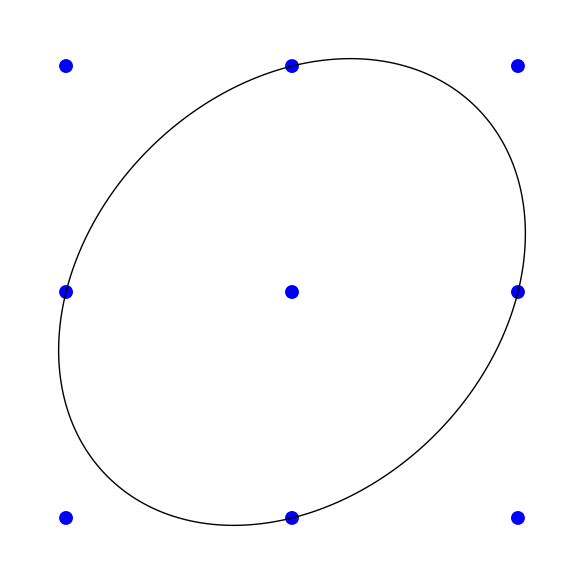} $\,\,\,\,$
    \includegraphics[scale=0.15]{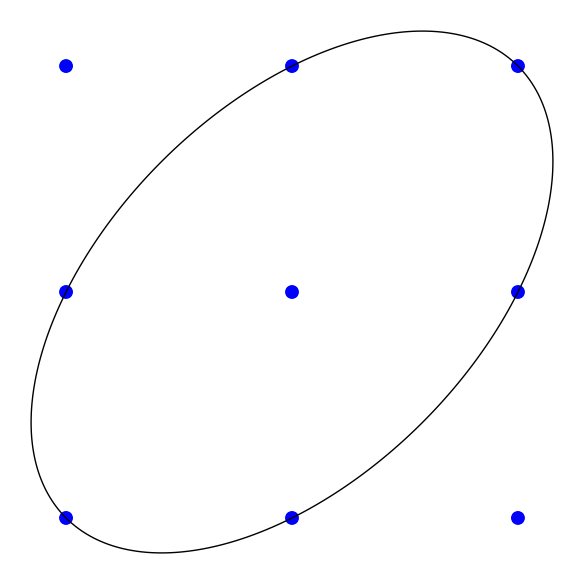}
  \end{center}
  \caption{A pencil of well-rounded forms.}
  \label{fig:pencil}
\end{figure}

To draw pictures in~$X$, we divide~$Z$ by a different homothety to
make its trace~$1$.  The image of~$X$ is then an open circular disc in
the $ab$-plane.  The well-rounded forms form a tree~$\widetilde{W}$
in~$X$, shown on the left in Figure~\ref{fig:wrr2T2}.  The vertical edge
in the center of the tree is the pencil in Figure~\ref{fig:pencil}.

\begin{figure}
  \begin{center}
    \includegraphics[scale=0.3]{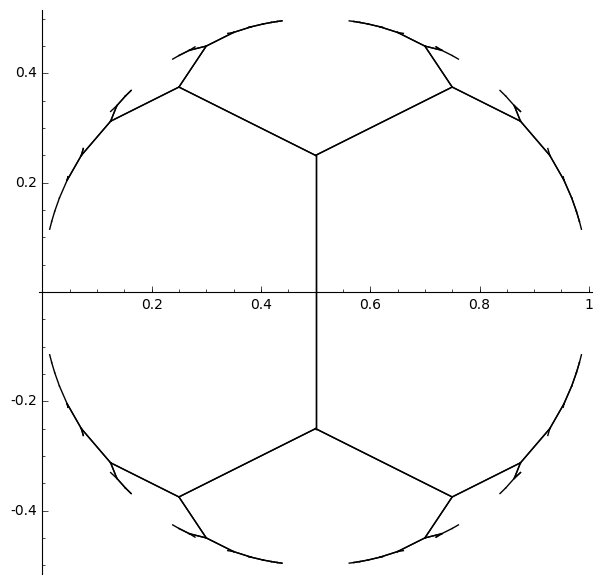} $\quad$
    \includegraphics[scale=0.3]{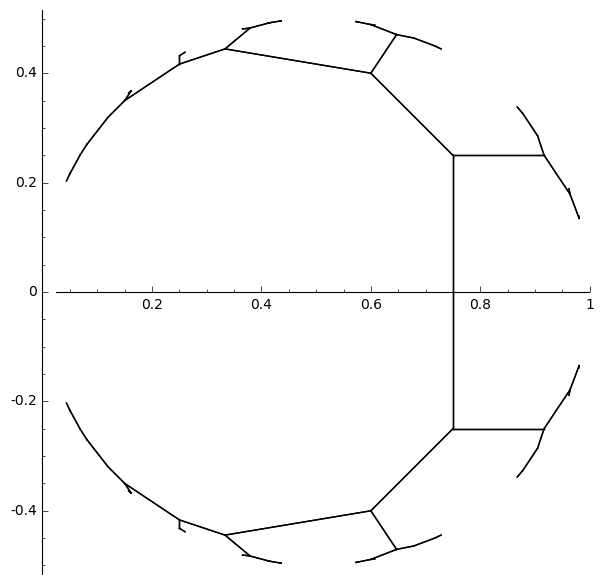}
  \caption{The well-rounded retract for $\GL_2(\Z)$, and its translate
    by $T_2$.}
  \label{fig:wrr2T2}
  \end{center}
\end{figure}

We can also understand~$Z$ in terms of lattices.  $X = G/K$, where $K$
is the subgroup of orthogonal matrices.  $Z = g g^t$ for some $g\in
G$.  The rows of~$g$ are a basis of~$\R^2$, and the $\Z$-linear
combinations of this basis are a lattice $L = L_0 g$.  Let $\Gamma_0 =
\GL_2(\Z)$, the integer matrices of determinant~$\pm1$.  Then $Y =
\Gamma_0\bs G$ is the space of lattices, and $Y/K = \Gamma_0\bs X$.
The tree~$\widetilde{W}$ is $\Gamma_0$-equivariant and is a strong
deformation retract of~$X$.  It is called the \emph{well-rounded
  retract}.

Let~$m$ be a positive integer.  The \emph{Hecke correspondence} $T_m$
is the one-to-many map that carries a lattice~$L$ to its sublattices
of index~$m$.
Figure~\ref{fig:sublattT2} shows the three sublattices for $\ell=2$.

\begin{figure}
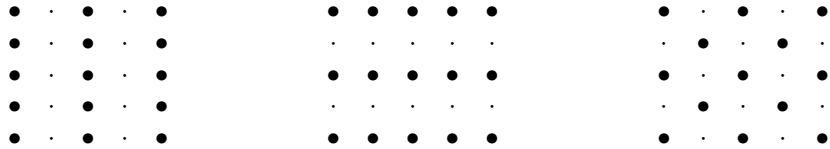

  \begin{center}
    \begin{footnotesize}
    $\begin{matrix}
      \bullet&\cdot&\bullet&\cdot&\bullet \\
      \bullet&\cdot&\bullet&\cdot&\bullet \\
      \bullet&\cdot&\bullet&\cdot&\bullet \\
      \bullet&\cdot&\bullet&\cdot&\bullet \\
      \bullet&\cdot&\bullet&\cdot&\bullet
    \end{matrix}
    \qquad\qquad\qquad
    \begin{matrix}
      \bullet&\bullet&\bullet&\bullet&\bullet \\
      \cdot&\cdot&\cdot&\cdot&\cdot \\
      \bullet&\bullet&\bullet&\bullet&\bullet \\
      \cdot&\cdot&\cdot&\cdot&\cdot \\
      \bullet&\bullet&\bullet&\bullet&\bullet
    \end{matrix}
    \qquad\qquad\qquad
    \begin{matrix}
      \bullet&\cdot&\bullet&\cdot&\bullet \\
      \cdot&\bullet&\cdot&\bullet&\cdot \\
      \bullet&\cdot&\bullet&\cdot&\bullet \\
      \cdot&\bullet&\cdot&\bullet&\cdot \\
      \bullet&\cdot&\bullet&\cdot&\bullet
    \end{matrix}$
    \end{footnotesize}
  \end{center}
  \caption{A lattice and its three sublattices of index two.}
  \label{fig:sublattT2}
\end{figure}

Let $\Gamma\subseteq\Gamma_0$ be a subgroup of finite index.  The
correspondence~$T_m$ acts on the cohomology $H^1(\Gamma\bs X)$, and
this action, also denoted~$T_m$, is a \emph{Hecke operator}.  By work
of Eichler and Shimura \cite{Sh}, the cohomology can be understood in
terms of modular forms (Eisenstein series and cusp forms) on
$\Gamma\bs\mathfrak{H}$.  The Hecke operators act on modular forms,
and their eigenvalues give important arithmetic information.

We would like to compute the Hecke operators by how they act on the
tree~$\widetilde{W}$.  Unfortunately, the Hecke correspondences do not
carry the tree to itself.  $T_2$ moves it within the disc, as shown on
the right in Figure~\ref{fig:wrr2T2}.  (The $T_m$ are one-to-many on
$\Gamma_0\bs X$ but are defined by a one-to-one map on~$X$; see
Section~\ref{subsec:defhecke} for the definition.)  The deformation
retraction $X \to \widetilde{W}$ will carry the right-hand tree back
to the left-hand tree, but when~$m$ is large this is hard to use for
computation.  One edge may retract back to a path through many edges,
starting and ending with fractions of edges.

This paper introduces the \emph{well-tempered complex
  $\widetilde{W}^+$ for $T_m$}.  For simplicity, let~$m$ be a
prime~$\ell$ for the rest of the Introduction.  The well-tempered
complex is a regular cell complex of dimension $\dim\widetilde{W} +
1$.  It is a fibration $\widetilde{W}^+ \to [1, \ell^2]$, where the
coordinate~$\tau$ in the base is called the \emph{temperament}.  The
fiber over~$\tau=1$ is on the left side of Figure~\ref{fig:wrr2T2},
with the fiber over~$\ell^2$ on the right.  The fibers are all trees,
and they continuously deform as~$\tau$ varies.  We show there are a
finite number of \emph{critical temperaments} where the cell structure
abruptly changes.  Figure~\ref{fig:wrr2T2near2} shows an example for
$T_2$, around $\tau=2$ in the base $[1,4]$.  All the fibers are
3-valent trees---every vertex is on three edges---except in the middle
picture, where for a moment some vertices are on four edges.

\begin{figure}
  \begin{center}
    \includegraphics[scale=0.2]{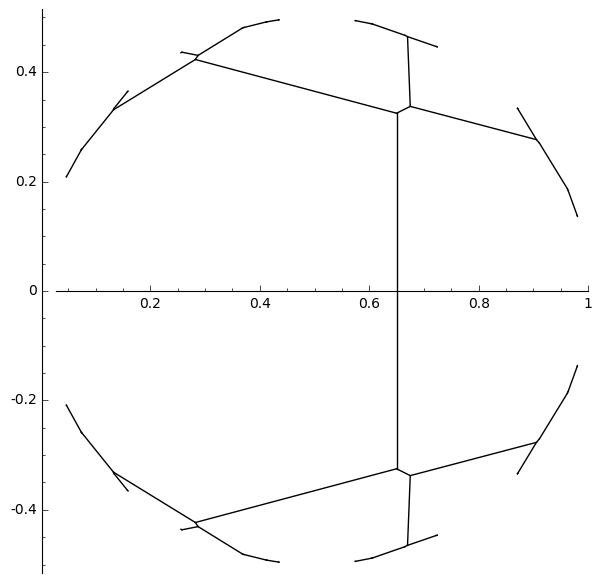}
    \includegraphics[scale=0.2]{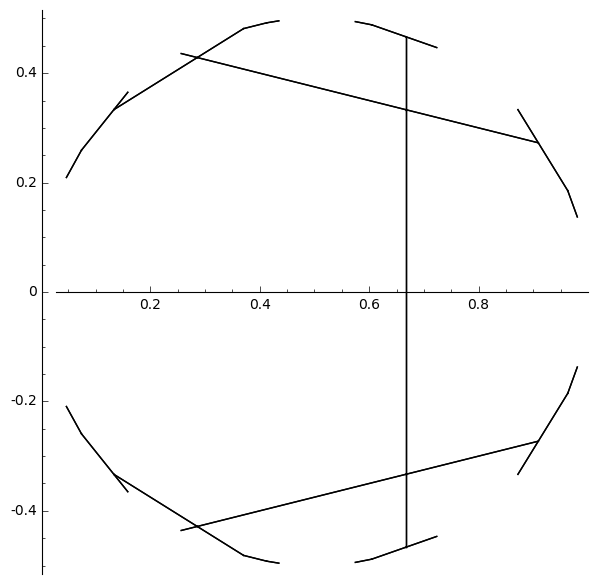}
    \includegraphics[scale=0.2]{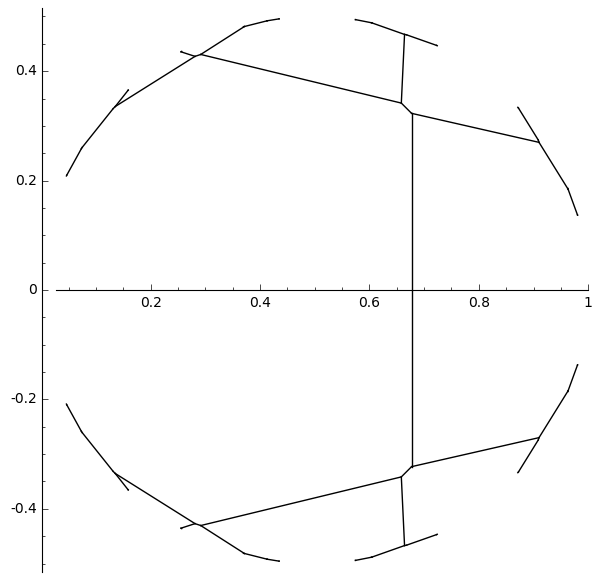}
  \end{center}
  \caption{Continuously deforming trees in the well-tempered complex.}
  \label{fig:wrr2T2near2}
\end{figure}

A deformation retract of $\Gamma\bs X$ is $\Gamma\bs \widetilde{W}_1$,
the finite graph which is the left side of Figure~\ref{fig:wrr2T2}
mod~$\Gamma$.  We compute $H^1(\Gamma\bs \widetilde{W}_\tau)$
as~$\tau$ varies from~$1$ to $\ell^2$.  At the critical temperaments,
the cohomology does not change, but the cell structure does, and we
keep track of it.  When~$\tau$ reaches the last temperament~$\ell^2$,
we have the information we need to compute~$T_\ell$ on the cohomology.
The algorithm is explained in more detail in
Section~\ref{sec:hecketope}.

\subsection{The General Case}

For $\Gamma\subseteq \GL_2(\Z)$, there is already a well-known method
for computing Hecke operators, \emph{modular symbols} \cite{Man}
\cite{St}.  It works on the top degree of cohomology,~$H^1$.  The
contribution of our paper is to give a Hecke operator algorithm on the
cohomology $H^i(\Gamma\bs X; \rho)$ for all $i=0, 1, 2, \dots$. This
holds for $\GL_n$ for all~$n\geqslant 2$.  Our algorithm works for all
local coefficient systems~$\rho$.

The general setting of this paper is the algebraic group~$\mathbb{G}$
where $\mathbb{G}(\Q)$ is the general linear group of a division
algebra~$D$ (possibly commutative) of finite dimension over~$\Q$.
When $D=\Q$, we have $\mathbb{G}(\Z) = \Gamma_0 = \GL_n(\Z)$, the case
described in the Introduction so far (for $n=2$).  In~\cite{Ash84},
Avner Ash constructed the well-rounded retract in this general setting
of division algebras.  In Section~\ref{sec:wrr}, we set up the rest of
the paper by reviewing notation and facts from~\cite{Ash84}.  The
retract depends on a family of weights; in Section~\ref{sec:varywts},
we show that varying the weights over an interval of temperaments
gives a fibration.

For any of the~$\mathbb{G}$ in this paper, the definition of a Hecke
correspondence starts from a sublattice $M_0 \subset L_0$ of finite
index.  (In our $2\times2$ example for $T_\ell$, the~$M_0$ is any one
of the index-$\ell$ sublattices of $\Z^2$.)  The form~$Z$ generalizes
to a Hermitian form~$Z$ which defines a Hermitian metric on~$L_0$.
The main idea of the well-tempered complex is to lie about the lengths
of vectors $x\in L_0 - M_0$.  We use a family of weights where the
weighted length of~$x$ is~$\tau$ times the length of~$x$ if $x \in L_0
- M_0$, but the weighted length of~$x$ equals the length of~$x$ if
$x\in M_0$.  When $\tau=1$, there are no lies, and we obtain the usual
well-rounded retract.  When~$\tau$ is large, vectors $x \in L_0 - M_0$
cannot have small weighted length: all the vectors of weighted
length~$1$ in a well-rounded lattice will be in~$M_0$.  For~$\tau$
between the extremes, the critical temperaments occur when the
changing weights make the cell structure vary.  This is the topic of
Section~\ref{sec:wtc}.

In Section~\ref{sec:hos}, we show that we can compute the Hecke
operators with the well-tempered complex.  In
Section~\ref{sec:hecketope}, we indicate how the different fibers of
the well-tempered complex are computed in practice.  This involves a
weighted generalization of the Voronoi polyhedron, the
\emph{Hecketope}.

The algorithm has been coded up for $\SL_n(\Z)$ for $n=2,3,4$ by the
first author.  It performs successfully for congruence subgroups of
small levels~$N$.  Section~\ref{sec:sl3ex} contains a sampling of
results for subgroups of $\SL_3(\Z)$.  In the near future, we hope to
test it for $\GL_2$ when~$D$ is a quadratic number field or another
small Galois extension of~$\Q$.

Our algorithm for Hecke operators has two parts, one-time work and
every-time work.  The one-time work is to compute the well-tempered
complex for a given $\Gamma_0$, $L_0$, and Hecke correspondence~$T$.
The every-time work is to compute the Hecke operators for~$T$ on
$H^i(\Gamma_0\bs X, \rho)$ for various~$i$ and~$\rho$.  To compute the
cohomology of subgroups~$\Gamma'$, like the classical $\Gamma_0(N)$
and $\Gamma(N)$ of level~$N$, we simply do the every-time work for the
coinduced representation $\mathrm{Coind}_{\Gamma'}^{\Gamma_0}\rho$.
Other Hecke-operator algorithms, including the classical algorithms
with modular symbols, require one to do the equivalent of the one-time
work and the every-time work every time.  Even for a single Hecke
correspondence, it is possible to run our algorithm as a parallel
computation (see Theorem~\ref{wtcalgor}).

Our current implementation is a proof of concept written in Sage
\cite{SageMath}.  In terms of speed, it is not yet competitive with
Gunnells' algorithm in \cite{Gun} \cite{AGM1}.  Plans for the future
include translating parts of the implementation to C{+}{+} and
parallelizing it.

\subsection{Related Work}

The Ash-Rudolph algorithm generalizes the modular symbol algorithm
\cite{AR}.  For any~$n$, it computes Hecke operators on $H^i(\Gamma\bs
X; \rho)$ for $\Gamma\subseteq\SL_n(\Z)$.  It works only in the top
degree, $i = \vcd\Gamma$.  Ash and his collaborators have many papers
that study cohomology using this algorithm, starting with~\cite{AGG}.

Gunnells extended the Ash-Rudolph algorithm to the next lower degree,
where $i = \vcd\Gamma - 1$.  This works for $\Gamma\subseteq\SL_n(\Z)$
for all~$n$ \cite{Gun}.  In a series of papers including \cite{AGM1}
\cite{AGM3} \cite{AGM7} \cite{AGM8}, Ash, Gunnells, and the first
author have computed Hecke operators on $H^5(\Gamma_0(N)\bs X; \rho)$,
which is in the cuspidal range of cohomology, for a variety of
levels~$N$ and coefficient systems~$\rho$.  From the Hecke
eigenvalues, these papers find Galois representations that are
conjecturally attached to the cohomology.

For other computations of Hecke operators on the cohomology of
subgroups of SL or GL over a number field, see for example
\cite{Men79} \cite{Cre84} \cite{Vog85} \cite{Ash86} \cite{vGvdKTV}
\cite{Rahm} \cite{CreAr} \cite{GGHSSY} \cite{AGMY} \cite{GMY}.

\subsection{Acknowledgments}

Avner Ash's paper \cite{Ash84} is foundational for this paper.  We
thank him for helpful conversations, as well as Paul Gunnells, Dan
Yasaki, Dylan Galt, and David Pollack.

We thank Tony Bahri for suggesting the name \emph{well tempered}.  See
also~\cite{DM}.  \emph{Soli Deo Gloria.}


\section{Background on the Well-Rounded Retract}
\label{sec:wrr}

We summarize notation and facts from~\cite{Ash84}.  We reverse left
and right, and rows and columns, relative to~\cite{Ash84} because we
prefer to have our arithmetic groups acting on the left on cell
complexes.

Let~$D$ be a division algebra (possibly commutative) of finite
dimension over~$\Q$.  Let $S = D \otimes_\Q \R$.  Let $G = \GL_n(S)$.
The space of row vectors $S^n$ is an $\R$-vector space, a left
$S$-module, and a right $G$-module.

A $\Z$-lattice in $D^n$ is a finitely generated subgroup of row
vectors that contains a $\Q$-basis of $D^n$.  A $\Z$-lattice in $S^n$
is a finitely generated subgroup that contains an $\R$-basis.  Let
$L_0 \subset D^n$ be a $\Z$-lattice.  The subset~$A$ of~$D$ that
stabilizes~$L_0$ is an order in~$D$.  The subgroup~$\Gamma_0$ of~$G$
that stabilizes~$L_0$ is an arithmetic group.  We fix~$L_0$, and hence
$A$ and~$\Gamma_0$, once and for all.

Fix an arithmetic subgroup $\Gamma\subseteq\Gamma_0$.  Let $Y =
\Gamma\bs G$.  If $\Gamma=\Gamma_0$, then~$Y$ is the space of
$\Z$-lattices in $S^n$ stable under~$A$ and isomorphic to~$L_0$ as
$A$-module, and these are the lattices of the form $L_0 g$ for $g\in
G$.  If~$\Gamma$ is a congruence subgroup, then the elements of~$Y$
are lattices with a level structure.  In~\cite{Ash84}, only~$\Gamma_0$
is used, but passing to~$\Gamma$ brings no real changes.

We choose a particular positive definite inner product $\langle \,,\,
\rangle$ on the $\R$-vector space $S^n$; we will present the
definition in Section~\ref{subsec:wrrpf}.  The $K\subset G$ which
preserves $\langle \,,\, \rangle$ is a maximal compact subgroup.  $X =
G/K$ is a symmetric space.  The virtual cohomological dimension
$\vcd\Gamma$ is $\dim(Y/K) - n$.

Let $\mathbb{P}(D^n)$ be the set of lines (rank-one left
$D$-submodules) in~$D^n$.  The quotient $\mathbb{P}(D^n)/\Gamma$ is
finite.  A \emph{set of weights} for~$\Gamma$ is a function $\varphi
: \mathbb{P}(D^n)/\Gamma \to \R_+$. Such a~$\varphi$ defines a
\emph{set of weights} for~$L_0$, a $\Gamma$-invariant function $L_0
- \{0\} \to \R_+$ also denoted~$\varphi$, by $\varphi(x) =
\varphi(Dx)$.  We always normalize~$\varphi$ by dividing through by a
positive real number so that the maximum value in its image is~$1$.
For $L = L_0 g$, a set of weights~$\varphi$ for~$L_0$ defines
$\varphi^L : L - \{0\} \to \R_+$, a \emph{set of weights} for~$L$, by
$\varphi^L(xg) = \varphi(x)$ for $x \in L_0 - \{0\}$.

Here are the main definitions and the main theorem of~\cite{Ash84}.

\begin{definition} \label{arithmin}
  Let $L = L_0 g \in Y$.  The \emph{arithmetic minimum} is
  \[
  m(L) = \min \{ \varphi^L(x) \langle x, x\rangle \mid x \in L -
  \{0\}\}.
  \]
  The
  \emph{minimal vectors} are
  \[
  M(L) = \{x \in L \mid \varphi^L(x)\langle x, x\rangle = m(L) \}.
  \]
  We say~$L$ is \emph{well rounded} if $S\,M(L) = S^n$.  The set of
  well-rounded lattices in~$Y$ with minimum~1 is
  denoted~$\widehat{W}$.
\end{definition}

The functions $m$ and~$M$ are $K$-invariant, because
$\langle\,,\,\rangle$ is.  Hence $\widehat{W}$ is $K$-invariant.

\begin{theorem}[{\cite[Thm.~2.11]{Ash84}}] \label{wrrmain}
  $W = \widehat{W}/K$ is a strong deformation retract of $Y/K$.  It is
  compact and of dimension $\vcd\Gamma$.  The universal
  cover\footnote{Strictly speaking, this is a ramified cover, because
    certain points of~$W$ have finite stabilizer subgroups
    in~$\Gamma_0$.  The barycentric subdivision in the last sentence
    of the theorem produces a triangulation which is compatible with
    the ramified covering map.}  $\widetilde{W}$ of~$W$ is a locally
  finite regular cell complex in~$X$ on which $\Gamma$ acts cell-wise
  with finite stabilizers of cells.  This cell structure has a natural
  barycentric subdivision which descends to a finite cell complex
  structure on~$W$.
\end{theorem}

\begin{definition}
  $W = \widehat{W}/K$ is the \emph{well-rounded retract}.
\end{definition}


\section{Varying the Weights}
\label{sec:varywts}

\subsection{Statement of Results}
In this paper, we start with the trivial bundle $Y \times I$ over an
interval~$I$.  The value $\tau\in I$ is the \emph{temperament}.  $G$
acts on $Y \times I$ by acting on each fiber $Y\times\{\tau\}$ in the
obvious way.  There is a bundle isomorphism $(Y\times I)/K \cong
(Y/K)\times I$ with fibers $Y/K$.

A \emph{one-parameter family of weights} for~$L_0$ is a map
$\varphi_\tau : (L_0 - \{0\}) \times I \to \R^+$ which is a
$\Gamma$-invariant set of weights for any given~$\tau$, and such that
$\varphi_\tau(x)$ is real analytic in~$\tau$ for any given~$x$.  We
divide through by a positive real scalar, which depends continuously
on~$\tau$, so that the maximum of $\varphi_\tau$ is~1 for all~$\tau$.
Why we use real analytic functions will appear during the proof of
Theorem~\ref{wtccont}.

A one-parameter family of weights $\varphi_\tau$ determines
$\varphi^L_\tau$, $m_\tau(L)$, $M_\tau(L)$, $\widehat{W}_\tau$, and
$W_\tau = \widehat{W}_\tau/K$ for any given~$\tau$. By
Theorem~\ref{wrrmain}, there is a strong deformation retraction
$R_\tau(t)$ of the fiber over~$\tau$ onto $W_\tau$.  This means
$R_\tau(t) : ((Y \times \{\tau\})/K) \times [0,1] \to (Y \times
\{\tau\})/K$ is continuous, $R_\tau(0)$ is the identity on the fiber
$(Y \times \{\tau\})/K$, the image of $R_\tau(1)$ is $W_\tau$, and
$R_\tau$ fixes $W_\tau$ pointwise for all~$t$.  Our next theorem is that
$R_\tau(t)$ is also continuous in~$\tau$.

\begin{theorem} \label{wtccont}
  $R_\tau(t)$ is a continuous function $((Y\times I)/K) \times [0,1]
  \to (Y\times I)/K$.
\end{theorem}

\begin{corollary}
  $\{(w \times \tau)/K \mid \tau \in I, w \in \widehat{W}_\tau\}$ is a
  strong deformation retract of $(Y \times I)/K$.  It has dimension
  $\vcd \Gamma$.  It is compact if~$I$ is compact.  The map from the
  retract to~$I$ is a fibration.
\end{corollary}

The proof of Theorem~\ref{wtccont} is an extension of the proof of
Theorem~\ref{wrrmain}.  In Section~\ref{subsec:wrrpf}, we summarize
notation and facts from the proof of Theorem~\ref{wrrmain}
in~\cite{Ash84}.  In Section~\ref{subsec:wtcpf} we prove
Theorem~\ref{wtccont}.

\subsection{The Well-Rounded Retraction} \label{subsec:wrrpf}
First, we recall the definition of the inner product
$\langle\,,\,\rangle$.  Let~$v$ run through the archimedean places of
the center of~$D$.  Let~$D_v$ be the corresponding completion; we
identify~$D_v$ with an algebra of square matrices over~$\R$, $\C$,
or~$\H$, as appropriate.  For $x\in D_v$, let $x^*$ be the conjugate
transpose, where \emph{conjugate} means the identity on~$\R$ and the
usual conjugate on~$\C$ or~$\H$.  By $\underline{x}$ we denote a row
vector.  Define the positive definite inner product
$\langle\,,\,\rangle_v$ on $D_v^n$ by $\langle \underline{x},
\underline{y} \rangle_v = \sum_{i=1}^n \Tr^{D_v}_\R (x_i^* y_i)$.  We
have $S = \prod_v D_v$ and $G = \prod_v G_v$, where $G_v =
\GL_n(D_v)$.  Define $\langle\,,\,\rangle$ on $S^n = \prod_v D_v^n$ by
$\langle (\underline{x}_v), (\underline{y}_v) \rangle = \sum_v \langle
\underline{x}_v, \underline{y}_v \rangle_v$.  The left action of~$S$
on~$S^n$ is stable under adjoint of $\R$-linear maps.  $K$ is the
subgroup of~$G$ preserving $\langle\,,\,\rangle$.  Observe that $K =
\prod_v K_v$, where $K_v$ is the maximal compact subgroup of $G_v$
preserving $\langle\,,\,\rangle_v$.  Also, $X = \prod_v X_v$ where
$X_v = G_v/K_v$.

Throughout this Section~\ref{subsec:wrrpf}, $\tau \in I$ is fixed.
For $j=1,\dots,n$, define
\[
Y_\tau^{(j)} = \{L \in Y_\tau \mid \dim_D (D\,M_\tau(L)) \geqslant j
\mathrm{\ and\ } m_\tau(L) = 1\}.
\]
Set $Y_\tau^{(0)} = Y$.  Then $Y_\tau^{(0)} \supset Y_\tau^{(1)}
\supset Y_\tau^{(2)} \supset\cdots$ is a nested sequence of subspaces
in the fiber~$Y$ over~$\tau$.  By definition, $\widehat{W}_\tau =
Y_\tau^{(n)}$.  There is a $K$-equivariant strong deformation
retraction $r_\tau^{(j)}(t)$ from $Y_\tau^{(j-1)}$ to $Y_\tau^{(j)}$
for $t\in[0,1]$, whose definition we give below at~\eqref{defrjtau}.
To obtain $R_\tau(t)$, we string together the $r_\tau^{(j)}(t)$ for
$j=1,\dots,n$, rescaling the~$n$ consecutive~$t$s into $[0,1]$, while
modding out by~$K$.  We write $r_\tau^{(j)}(t)$ as $r_\tau^{(j)}(L,t)$
when we need to specify~$L$.

The action of the positive real homotheties gives~$r_\tau^{(1)}$.
Suppose $j\geqslant 2$ and $L\in Y_\tau^{(j-1)}$.  We now define
$r_\tau^{(j)}(L,t)$.  If $L$ is in $Y_\tau^{(j)}$ already, set
$r_\tau^{(j)}(L,t) = L$ for all~$t$.  Otherwise, let $P = S\cdot
M_\tau(L)$, an $S$-submodule of~$S^n$.  Let~$Q$ be the orthogonal
complement of~$P$ in~$S^n$ with respect to~$\langle\,,\,\rangle$.  By
\cite[Lemma~2.6]{Ash84}, $Q$ is also an $S$-submodule, $S^n = P \oplus
Q$, and if~$\lambda$ is a non-zero real number then the map
$F_\lambda(p+q) = p+\lambda q$ (for $p\in P$, $q\in Q$) is given by
right multiplication by some matrix $F_\lambda \in G$.  The heart of
the argument is to prove that $\mu_\tau(L) = \sup\{\lambda \mid
m_\tau(L\,F_\lambda) < 1\}$ exists and satisfies $0 < \mu_\tau(L) <
1$, with $m_\tau(L\cdot F_\lambda) = 1$ at $\lambda = \mu_\tau(L)$.
(Intuitively, we shrink~$\lambda$ down from~1 through positive values
until $L\,F_\lambda$ gains at least one new independent minimal
vector; this new vector appears at $\lambda = \mu_\tau(L)$.)  Define
\begin{equation} \label{defrjtau}
r_\tau^{(j)}(L,t) =
\left\{
\begin{array}{cl}
  L\,F_{(1 + (\mu_\tau(L) - 1)t)} &
  \quad\mathrm{if\ } L \in Y_\tau^{(j-1)} - Y_\tau^{(j)} \\
  L &
  \quad\mathrm{if\ } L \in Y_\tau^{(j)}.
\end{array}
\right.
\end{equation}
This is $K$-invariant, and~\cite{Ash84} proves it is continuous in~$L$
and~$t$.

\subsection{Proof of Theorem~\ref{wtccont}}  \label{subsec:wtcpf}
Let $Y^{(j)} \subseteq Y\times I$ be the union of the $Y_\tau^{(j)}$
for $\tau\in I$.  Let $L\in Y^{(j-1)}$.  We must show
$r_{\tau'}^{(j)}(L,t)$ is continuous as~$\tau'\in I$ varies in a
neighborhood of~$\tau$.  We will show it is continuous for both $\tau'
\in (\tau-\varepsilon, \tau]$ and $\tau' \in [\tau, \tau+\varepsilon)$
    for small $\varepsilon > 0$.  We will call either
    $(\tau-\varepsilon, \tau]$ or $[\tau, \tau+\varepsilon)$ a
    \emph{half-interval around}~$\tau$.  A \emph{punctured
      half-interval} is a half-interval around~$\tau$ with~$\tau$
    removed.

As above, the action of the positive real homotheties gives the
retraction $r^{(1)}$.  This is continuous as a function of~$\tau$
because $\varphi_\tau$ is continuous in~$\tau$ and all the weights are
positive.  From now on, let $j\geqslant 2$.

The dimension $\dim_D (D\,M_\tau(L))$ may or may not be constant on
small half-intervals.  If there is no small half-interval on a given
side on which it is constant, it nevertheless varies upper
semi-continuously: the dimension has a certain value at~$\tau$, and
has a constant but smaller value on a small punctured half-interval.
The reason for the upper-semicontinuity is that $\dim_S
(S\,M_\tau(L))$ is the number of linearly independent vectors $x\in L$
with
\begin{equation} \label{usc}
\varphi_\tau^L(x) \langle x,x \rangle - m_\tau(L) = 0.
\end{equation}
For a given~$x$, \eqref{usc} is an upper semi-continuous condition,
because $\varphi_\tau^L(x)$ is a real analytic positive function
of~$\tau$, $m_\tau(L)$ is a min of real analytic functions of~$\tau$,
and $\langle x, x \rangle$ is a positive constant independent
of~$\tau$.  If the left side of~\eqref{usc} is~$\ne 0$, then it will
remain $\ne 0$ on a small half-interval around~$\tau$, by continuity.
But if it is~$0$, then---by analyticity---it will either remain~$0$ on
a small half-interval around~$\tau$, or it will be~$\ne 0$ throughout
a small punctured half-interval around~$\tau$.

Showing $r_\tau^{(j)}(L,t)$ is continuous in~$\tau'$ breaks into two
cases.  In the first case, $\dim_S P$ is constant for~$\tau'$ in a
small half-interval around~$\tau$.  This is equivalent to saying $L
\in Y_{\tau'}^{(j-1)}$ for all~$\tau'$ in the neighborhood.  Then~$P$
and $Q$ vary continuously with $\tau'$ as $S$-submodules of~$S^n$,
hence so does $L\,F_\lambda$.  The minimum $m_{\tau'}$ varies
continuously with~$\tau'$, so $\mu_{\tau'}(L)$ varies continuously
with~$\tau'$ and remains bounded away from~1 on the half-interval.
Hence $r_\tau^{(j)}(L,t)$ varies continuously with~$\tau'$ throughout
the half-interval.

In the second case, $\dim_S (S\,M_{\tau'}(L))$ is upper
semi-continuous.  Here $L \notin Y_{\tau'}^{(j-1)}$ for $\tau'$ on a
small punctured half-interval around~$\tau$; rather, there is a $j' <
j$ so that $L \in Y_{\tau'}^{(j'-1)}$ for $\tau'$ in that punctured
half-interval.  The retraction $r^{(j)}$ will not be operating on $L$
when $\tau'\ne\tau$, because $L \notin Y_{\tau'}^{(j-1)}$; rather, it
will be operating on the image~$L'$ of~$L$ under $r^{(j')}$,
$r^{(j'+1)}, \dots, r^{(j-1)}$.  We must show that~$L'$ is arbitrarily
close to~$L$ when~$\tau'$ is arbitrarily close to~$\tau$.  Consider
the function~$\mu$ used in~\eqref{defrjtau} above.  When $\tau'$ is
very close to~$\tau$, the new minimal vectors that will appear when
$\lambda = \mu_{\tau'}(L)$ are a subset of those that are already in
$M_\tau(L)$.  This means that $\mu_{\tau'}(L) \to 1$ as $\tau' \to
\tau$.  In~\eqref{defrjtau}, this means $L\,F_{(1 + (\mu_\tau(L) -
  1)t)}$ converges to~$L$ as $\tau' \to \tau$.  The convergence is
uniform as~$L$ and~$t$ vary in small compact neighborhoods.  This
proves the continuity of $r_\tau^{(j)}(L,t)$.


\section{The Well-Tempered Complex}
\label{sec:wtc}

\subsection{Hecke Correspondences}
\label{subsec:defhecke}
We review Hecke correspondences, following \cite[\S3.1 and p.~76]{Sh}.
Subgroups~$\Gamma_1$ and~$\Gamma_2$ of~$G$ are \emph{commensurable} if
$\Gamma_1\cap\Gamma_2$ has finite index in both~$\Gamma_1$
and~$\Gamma_2$.  Commensurability is an equivalence relation.  The
\emph{commensurator} of~$\Gamma_1$ is $\breve\Gamma_1 = \{a\in G \mid
a^{-1} \Gamma_1 a$ is commensurable with~$\Gamma_1\}$.  If~$\Gamma_1$
and~$\Gamma_2$ are commensurable, then $\breve\Gamma_1 =
\breve\Gamma_2$.

Define $\Delta = \{a \in G \mid L_0 a \subseteq L_0\}$.  This is a
sub-semigroup of~$G$ containing~$\Gamma_0$.  If $a\in\Delta$, then
$M_0 = L_0 a$ is a sublattice of~$L_0$ of finite index.  Thus
$\Gamma_0 \subset \Delta \subset \breve\Gamma_0$.  The arithmetic
group $\Gamma = \Gamma_0 \cap a^{-1} \Gamma_0 a$ is the common
stabilizer in~$G$ of $L_0$ and $M_0$.  One calls $(\Gamma_0, \Delta)$
a \emph{Hecke pair}.

A point in $\Gamma_0\bs X$ has the form $\Gamma_0 g K$ with $g\in G$.
Define two maps
\begin{equation} \label{twomaps}
\xymatrix{
  \Gamma\bs X \ar@/_/[d]_p \ar@/^/[d]^q \\
  \Gamma_0\bs X
}
\end{equation}
by $p : \Gamma g K \mapsto \Gamma_0 g K$ and $q : \Gamma g K \mapsto
\Gamma_0 a g K$.  The map~$q$ is well defined: $ (\forall\,\gamma \in
\Gamma) \, \Gamma_0 a g = \Gamma_0 a\gamma g \mathrm{\ }$ is true if
and only if $\Gamma_0 = \Gamma_0 a\gamma a^{-1}$, and the latter is
implied by $a\gamma a^{-1} \in a(\Gamma_0 \cap a^{-1} \Gamma_0 a)
a^{-1} = a \Gamma_0 a^{-1} \cap \Gamma_0 \subseteq \Gamma_0$.

The \emph{Hecke correspondence} $T_a$ is the one-to-many map
$\Gamma_0\bs X \to \Gamma_0\bs X$ given by
\[
T_a = q \circ p^{-1}.
\]
It sends one point of $\Gamma_0\bs X$ to $[\Gamma_0 : \Gamma]$ points
of $\Gamma_0\bs X$, counting multiplicities.

The \emph{Hecke algebra} for the Hecke pair $(\Gamma_0, \Delta)$ is
the free abelian group on the set of correspondences~$T_a$ for
$a\in\Delta$, with multiplication defined by the composition of the
correspondences.  This is equivalent to the traditional definition as
the algebra of double cosets $\Gamma_0 a\Gamma_0$ for $a\in\Delta$
\cite[p.~54]{Sh}.  The Hecke algebra is commutative.

\subsubsection{Examples}
\label{subsubsec:defhecke1}

Let $D=\Q$, and let~$L_0$ be the standard lattice $\Z^n$.  Then
$\Gamma_0 = \GL_n(\Z)$.  All lines in $\mathbb{P}^n(\Q)$ are
equivalent mod~$\Gamma_0$, so the constant function $\varphi=1$ is the
only set of weights for~$\Gamma_0$.  Here~$\Delta$ is the semigroup of
all $n\times n$ matrices with entries in~$\Z$ and non-zero
determinant.  $\breve\Gamma_0 = \GL_n(\Q)$.  For a prime $\ell\in\Z$
and $k\in\{1,\dots,n\}$, define
\[
T_{\ell,k} = T_a \quad\mathrm{for\ } a =
\mathrm{diag}(\underbrace{1,\dots,1}_{n-k\mathrm{\ times}},
\underbrace{\ell,\dots,\ell}_{k\mathrm{\ times}}).
\]
The Hecke algebra is generated by the $T_{\ell,k}$ for all
primes~$\ell$ and $k\in\{1,\dots,n\}$ \cite[\S3.2]{Sh}.


When $D$ is an algebraic number field of class number one, the Hecke
algebra is defined in a similar way.  The ring of integers
$\mathcal{O}_D$ replaces~$\Z$.

When~$D$ is a number field of class number greater than one, what we
have defined is a subalgebra of the full Hecke algebra, namely the
subalgebra consisting of those correspondences that come from the
principal ideal classes.  The full Hecke algebra is defined over the
adeles; see, for example, \cite[p.~47]{Gel}.

\subsection{Definition of the Well-Tempered Complex}

Fix a set of weights~$\varphi$ for~$\Gamma_0$.  The lattice~$L_0$,
and~$\varphi$, determine the well-rounded retract for~$\Gamma_0$ as in
Section~\ref{sec:wrr}.  Fix $a\in\Delta$, and let $\Gamma = \Gamma_0
\cap a^{-1} \Gamma_0 a$ as in Section~\ref{subsec:defhecke}.  The
well-tempered complex~$W^+$ will be determined by~$L_0$, $\varphi$,
and~$a$, and will naturally have an action of~$\Gamma$.

We impose an additional hypothesis on~$\varphi$ making~$\varphi$
compatible with~$a$:
\begin{equation} \label{phi_a}
  \varphi(xa) = \varphi(x) \qquad\mathrm{for\ all\ } x \in L_0 -
  \{0\}.
\end{equation}
The hypothesis is satisfied, for instance, if~$\varphi$ is the
constant function~$1$.

Let $M_0 = L_0 a$.  For the rest of the paper, we extend the set of
weights~$\varphi$ for~$\Gamma_0$ to a set of weights~$\varphi_\tau$
for~$\Gamma$ in a particular way.

\begin{definition} \label{deftausq}
  For $x \in L_0 - \{0\}$ and $\tau \geqslant 1$, define
\[
\varphi_\tau(x) = \left\{
\begin{array}{cl}
  \varphi(x) & \quad\mathrm{if\ }x \in M_0 - \{0\}, \\
  \tau^2\varphi(x) & \quad\mathrm{if\ }x \notin M_0.
\end{array} \right.
\]
\end{definition}

\begin{remark}
The idea here comes from $m(L)$ in Definition~\ref{arithmin}.  The
weighted squared length of a vector $x \in L$ is $\varphi^L(x) \langle
x, x \rangle$.  The squared length $\langle x, x \rangle$ scales by
$c^2$ when we multiply~$x$ by $c\in\R$.  By multiplying the weight
by~$\tau^2$ when $x \notin M_0$, we mimic the effect of scaling the
length of~$x$ linearly by~$\tau$.  We pretend $x\notin M_0$ gets
longer by lies, linearly.  When $x\in M_0$, we do not pretend to
change the length.
\end{remark}

Choose $\tau_0 > 1$, and let $I = [1, \tau_0]$.  The well-tempered
complex
depends on~$\tau_0$, but we will see that the complexes for two
different~$\tau_0$ are isomorphic when~$\tau_0$ is sufficiently large.
In Theorem~\ref{wtctau0} below, we will derive a particular $\tau_0 >
1$ for computing Hecke operators.

\begin{definition}
  The \emph{well-tempered complex}~$W^+$ for~$L_0$, $\varphi$, and~$a$
  is the image of $(Y \times [1,\tau_0])/K$ under the retraction
  $R_\tau(t)$ of Theorem~\ref{wtccont}, where~$\varphi$
  satisfies~\eqref{phi_a} and where~$\varphi_\tau$ is as in
  Definition~\ref{deftausq}.
\end{definition}

\subsection{Cells in the Well-Tempered Complex}
\label{wtccellpf}

\begin{theorem} \label{wtccell}
  The universal cover~$\widetilde{W}^+$ of the well-tempered complex~$W^+$
  is a locally finite regular cell complex on which $\Gamma$ acts
  cell-wise with finite stabilizers of cells.  This cell structure has
  a natural barycentric subdivision which descends to a finite cell
  complex structure on~$W^+$.
\end{theorem}

\begin{proof}
  We summarize notation and facts from the proof of
  Theorem~\ref{wrrmain} in~\cite{Ash84}, at first with~$\tau$ fixed.
  For any ring~$R$, let $M_n(R)$ be the ring of $n\times n$ matrices
  over~$R$, and $(\dots)^t$ the transpose.  For $Z_v \in M_n(D_v)$,
  let $Z_v^*$ be the result of applying ${}^*$ to each matrix entry
  of~$Z_v$.  We say $Z_v \in M_n(D_v)$ is \emph{Hermitian} if $Z_v^t =
  Z_v^*$.  Let $E_v$ be the $\R$-vector space of Hermitian matrices in
  $M_n(D_v)$.  $Z_v$ is \emph{positive definite} if $\Tr^{D_v}_\R
  \underline{x} Z_v \underline{x}^* > 0$ for all $\underline{x} \in
  D_v^n - \{0\}$.  Identify $X_v = G_v/K_v$ with the open cone of
  positive definite matrices in~$E_v$, via $g K_v \leftrightarrow g
  (g^*)^t$ for $g\in G_v$.  Let $E = \prod_v E_v$, an $\R$-vector
  space, and identify points $X = \prod_v X_v$ with products $Z =
  \prod_v Z_v$ of positive definite matrices.  If $x\in L_0$ and $Z =
  g (g^*)^t$ for $g\in G$, the length $\langle xg, xg \rangle$ is
  denoted $Z[x]$.

As mentioned in the Introduction, $E$ puts a system of \emph{linear
  coordinates} on~$X$.  If $x \in L_0$ is fixed, then $Z[x]$ is an
$\R$-linear function of~$Z$.  The cells in our complexes will be
defined using this linear structure.

The definitions of $m(L)$, $M(L)$, and \emph{well rounded} from
Definition~\ref{arithmin} are equivalent to the following for $Z\in
X$:
\begin{gather}
  m_\tau(Z) = \min\{\varphi_\tau(x) Z[x] \mid x \in L_0 - \{0\}\}
  \label{wrrdefW} \\
  M_\tau(Z) = \{x \in L_0 \mid \varphi_\tau(x) Z[x] = m_\tau(Z)\} \notag \\
  Z \mathrm{\ is\ well\ rounded\ if\ } S\,M_\tau(Z) = S^n. \notag
\end{gather}
If~$\widetilde{W}_\tau$ is the set of well-rounded~$Z$ in~$X$ with
minimum~$m_\tau(Z) = 1$, then~$\widetilde{W}_\tau$ is the inverse
image of~$W_\tau$ under the covering map $X \to Y/K$; this is the
universal cover of~$W_\tau$.

A point $Z \in \widetilde{W}_\tau$ satisfies the linear equations and
inequalities
\begin{align*}
  \varphi_\tau(x) Z[x] = 1 &\quad\mathrm{if\ } x \in M_\tau(Z), \\
  \varphi_\tau(x) Z[x] \geqslant 1 &\quad\mathrm{if\ } x \in L_0 - \{0\}.
\end{align*}
For each finite subset $M \subseteq L_0$, we define $\sigma_\tau(M)$ to be
the set of all $Z\in X$ such that
\begin{align}
  \varphi_\tau(x)Z[x] = 1 &\quad\mathrm{for\ }x \in M \label{wtcsigma} \\
  \varphi_\tau(x)Z[x] > 1 &\quad\mathrm{for\ }x \in L_0 - (M \cup \{0\}). \notag
\end{align}
This is a system of linear equations and inequalities in the linear
coordinates of~$E$.  (Intuitively, it is a linear programming problem
in~$E$.)  Either $\sigma_\tau(M)$ is empty, or it is a convex
polyhedron, which is topologically a closed cell.  When the cell is
non-empty, we call~$M$ the set of \emph{minimal vectors} for
$\sigma_\tau(M)$. \cite{Ash84} shows that the non-empty
$\sigma_\tau(M)$ give~$\widetilde{W}_\tau$ the structure of a locally
finite regular cell complex on which $\Gamma_0$ acts cell-wise with
finite stabilizers of cells.  It shows that the cell structure has a
natural barycentric subdivision which descends to a finite cell
complex structure on~$W_\tau$.  An element $\gamma\in\Gamma_0$ acts by
$\gamma \sigma_\tau(M) = \sigma_\tau(M \gamma^{-1})$.

We now consider how these constructions vary with~$\tau$.  For each
finite subset $M \subseteq L_0$, extend~\eqref{wtcsigma} to the
following system of equations and inequalities in~$Z$ and~$\tau$:
\begin{align*}
  \varphi_\tau(x)Z[x] = 1 &\quad\mathrm{for\ }x \in M, \\
  \varphi_\tau(x)Z[x] > 1 &\quad\mathrm{for\ }x \in L_0 - (M \cup \{0\}), \\
  1 \leqslant \tau \leqslant \tau_0. \notag
\end{align*}
By Definition~\ref{deftausq}, this is equivalent to
\begin{align}
  \varphi(x)Z[x] = 1 \label{wtcsigmatau2}
  &\quad\mathrm{for\ }x \in M \cap M_0, \\
  \varphi(x)Z[x] = \frac{1}{\tau^2}
  &\quad\mathrm{for\ }x \in M - M_0, \notag \\
  \varphi(x)Z[x] > 1
  &\quad\mathrm{for\ }x \in M_0 - (M \cup \{0\}), \notag \\
  \varphi(x)Z[x] > \frac{1}{\tau^2}
  &\quad\mathrm{for\ }x \in L_0 - (M_0 \cup M), \notag \\
  1 \leqslant \tau \leqslant \tau_0. \notag
\end{align}

We now add one more variable to our system of linear coordinates.
Replace~$E$ with the real vector space $E\times \R$, where the
coordinate on the second factor is~$u$.  Rewrite~\eqref{wtcsigmatau2}
on $E\times \R$ by setting
\[
u = \frac{1}{\tau^2}.
\]
In the first four equations of~\eqref{wtcsigmatau2}, $u$ occurs (if at
all) only to the first power.  The last equation bounds~$u$ between
two constants, $1/\tau_0^2$ and~$1$.  Thus the equations and
inequalities~\eqref{wtcsigmatau2} \emph{are still linear} in the
coordinates of~$Z$ and~$u$, still a linear programming problem.
Exactly as in~\cite{Ash84}, the solution set for~$M$ is either empty
or is a convex polyhedron $\sigma(M)$, which is topologically a closed
cell.  Taken over all~$M$, the $\sigma(M)$ give~$\widetilde{W}^+$ the
structure of a regular cell complex on which $\Gamma$ acts cell-wise
with finite stabilizers of cells.  The cell structure has a natural
barycentric subdivision which descends to a cell complex structure
on~$W^+$.

We must show there are only finitely many cells~$\sigma$ up to the
action of~$\Gamma$.  Since a convex polyhedron is the convex hull of
its vertices, it suffices to show there are only finitely many
zero-dimensional cells (vertices)~$\sigma$ in $\widetilde{W}^+$ up to
the action of~$\Gamma$.  Fix a $\Z$-basis of~$L_0$.  Given $b\in\R$,
define the \emph{box for}~$b$ to be the set of vectors in~$L_0$ whose
coordinates with respect to the basis all have absolute
value~$\leqslant b$.  Let $\Gamma'$ be the intersection of the
conjugates of~$\Gamma$ in~$\Gamma_0$; this is a normal subgroup
of~$\Gamma_0$ of finite index.  By the finiteness for the ordinary
well-rounded retract~$\widetilde{W}_1$ mod~$\Gamma'$, there is a~$b_0$
so that every vertex of~$\widetilde{W}_1$ has, modulo~$\Gamma'$, the
form $\sigma_1(M)$ for an~$M$ contained in the box for~$b_0$.  Then
every vertex that is a solution to~\eqref{wtcsigmatau2} with
$1\leqslant\tau\leqslant\tau_0$ has, modulo~$\Gamma'$, the form
$\sigma(M)$ for an~$M$ contained in the box for~$b_0\tau_0$.  The box
has only finitely many subsets.
\end{proof}

\emph{Notation.} For any $\tau \geqslant 1$, we write $u = 1/\tau^2$
for the rest of the paper.  We write $W_u = W_{1/\tau^2}$, and
similarly for $\widetilde{W}_u$.

\subsection{Refining the Well-Tempered Complex}
\label{subsec:refining}

In $\widetilde{W}_\tau$ for a given~$\tau$, the non-empty cells
$\sigma_\tau(M)$ are in one-to-one correspondence with their sets of
minimal vectors~$M$.  In~$\widetilde{W}^+$, the cells~$\sigma$ are
also indexed by a finite amount of combinatorial data.  For
computation in~$\widetilde{W}^+$, it will be more convenient to take a
certain refinement of the cell structure.  In the refinement, the
combinatorial data in~$L_0$ will be kept separate from the
combinatorial data along the $\tau$-axis.

Since a non-empty~$\sigma$ is a closed convex polyhedron, its
projection onto the $u$-coordinate is a compact subinterval of
$[1/\tau_0^2, 1]$, possibly a single point.  Consider the
set~$\mathcal{U}$ of left and right endpoints of these intervals,
taken over all~$\sigma$.  Because the cells are convex, the endpoints
are the images of certain vertices of the~$\sigma$s.  By the
finiteness in Theorem~\ref{wtccell}, $\mathcal{U}$ is a finite subset
of $[1/\tau_0^2, 1]$.  We denote its elements
\begin{equation} \label{defcrittemp}
1/\tau_0^2 = u^{(0)} < u^{(1)} < \cdots < u^{(i_r)} = 1.
\end{equation}
The $u^{(i)}$, and their corresponding $\tau^{(i)} =
1/\sqrt{u^{(i)}}$, are called the \emph{critical temperaments}.  We
will always refine the cells~$\sigma$ of Theorem~\ref{wtccell} by
cutting them into closed pieces along the hyperplanes $u = u^{(i)}$
for $i = 0, \dots, i_r$.  Each non-empty cell of the refinement is
indexed by a pair.  The pair is $(M, [u^{(i-1)}, u^{(i)}])$ if the
projection of the cell to the $u$-coordinate is $[u^{(i-1)},
  u^{(i)}]$.  The pair is $(M, [u^{(i)}, u^{(i)}])$ if the projection
is $\{u^{(i)}\}$.  We will write $[u, u']$ as shorthand for both
$[u^{(i-1)}, u^{(i)}]$ and $[u^{(i)}, u^{(i)}]$.

\subsection{The First and Last Temperaments}

Since~$M_0$ is a sublattice of finite index in~$L_0$, there is a
positive integer~$N$ so that $N L_0 \subseteq M_0 \subseteq L_0$.
Let~$\xi$ be the maximum value of~$\varphi$ on $x \in M_0 - \{0\}$.
Let~$\eta$ be the minimum value of~$\varphi$ on $x\in L_0 - M_0$.

\begin{theorem} \label{wtctau0}
  Let~$L_0$, $\varphi$, and~$a$ be as in Theorem~\ref{wtccell}.
  Define
  \[
  \tau_0 = \sqrt{\xi/\eta} \, N.
  \]
  Let $\tau\geqslant\tau_0$.  Then the map $X \to X$ given by $gK
  \mapsto a^{-1} gK$ descends mod~$\Gamma$ to give a cell-preserving
  homeomorphism from the well-rounded retract $W_1$ over~$1$ to the
  well-rounded retract $W_\tau$ over~$\tau$.  If a cell over $\tau =
  1$ is $\sigma_1(Q)$ with index set $Q \subset L_0 - \{0\}$, then the
  cell that corresponds to $\sigma_1(Q)$ under the homeomorphism has
  index set $Qa$.
\end{theorem}

We call the endpoints of $[1, \tau_0]$ the \emph{first} and
\emph{last} temperaments, respectively.  For example, if $D=\Q$ and
$T_a = T_{\ell,k}$ as in Section~\ref{subsubsec:defhecke1}, then
$\tau_0 = \ell$.

We turn to the proof of Theorem~\ref{wtctau0}.

\begin{lemma}
  With notation as in Theorem~\ref{wtctau0}, let $\tau\geqslant
  \tau_0$.  Let $L = L_0 g$ be a well-rounded lattice in the
  fiber over~$\tau$.  Then every minimal vector $x \in M_\tau(L)$ is
  in~$M_0 g$.
\end{lemma}

\begin{proof}[Proof of Lemma]
  First suppose $\tau > \tau_0$.  Assume by way of contradiction that
  $x \in M_\tau(L)$ is a minimal vector with $x \notin M_0 g$.  Note
  that $Nx \in M_0 g$.  The weighted length of~$x$ is $\varphi_\tau(x)
  \langle x, x \rangle \geqslant \tau^2 \eta \langle x, x \rangle$.
  The weighted length of~$Nx$ is $\varphi_\tau(Nx) \langle Nx, Nx
  \rangle \geqslant \xi N^2 \langle x, x \rangle$.  Since~$x$ was
  minimal, we have $\tau^2 \eta \leqslant \xi N^2$, or $\tau \leqslant
  \sqrt{\xi/\eta} N$, contradicting $\tau > \tau_0$.

  When $\tau = \tau_0$, the argument needs to be modified because
  there may be ties.  $\tau^2 \eta = \xi N^2$ may occur for an~$x$
  with $x \notin M_0 g$ but $Nx \in M_0 g$.  This says that two of the
  conditions in~\eqref{wtcsigmatau2}, those for~$x$ and~$Nx$, are
  identical and redundant in the fiber over this particular~$\tau_0$.
  We simply drop~$x$ from the minimal vector set and keep~$Nx$.
\end{proof}

\begin{proof}[Proof of Theorem~\ref{wtctau0}]
Let $g\in G$ be such that $L_0 g K \in W_\tau$.  We use the notation
$Z = g (g^*)^t$ of the proof of Theorem~\ref{wtccell}.  To say $L_0 g
K$ is well rounded over~$\tau$ is to say that there is a finite subset
$M \subset L_0 - \{0\}$ so that
\begin{align*}
  \varphi_\tau(x) Z[x] = 1
  &\quad\mathrm{for\ } x \in M \\
  \varphi_\tau(x) Z[x] > 1
  &\quad\mathrm{for\ } x \in L_0 - (M \cup \{0\}).
\end{align*}
By the Lemma, $M$ is a subset of $M_0 - \{0\}$; every $x\in M$ has the
form $x = ya$ for some $y \in L_0 - \{0\}$.  Let~$Q$ be the set of
these~$y$, so that $M = Qa$.  Let $Z_1 = a Z (a^*)^t$.  Then
\begin{align*}
  \varphi_\tau(x) Z_1[y] = 1
  &\quad\mathrm{for\ } y \in Q \\
  \varphi_\tau(x) Z_1[y] > 1
  &\quad\mathrm{for\ } y \in L_0 - (Q \cup \{0\})
\end{align*}
with $x=ya$.  Since $x\in M_0$, Definition~\ref{deftausq} says
$\varphi_\tau(x)$ is independent of~$\tau$, hence equals
$\varphi_1(x)$.  By the hypothesis~\eqref{phi_a}, $\varphi_1(x) =
\varphi_1(y)$.  Thus
\begin{align*}
  \varphi_1(y) Z_1[y] = 1
  &\quad\mathrm{for\ } x \in Q \\
  \varphi_1(y) Z_1[y] > 1
  &\quad\mathrm{for\ } x \in L_0 - (Q \cup \{0\}).
\end{align*}
In other words, $Z_1$ is well rounded in the fiber over~$1$.

It is easy to see this map is a homeomorphism between the well-rounded
retracts in the two fibers.  Since the map from the fiber over~$\tau$
to the fiber over~$1$ was multiplication by~$a$, the map from the
fiber over~$1$ to the fiber over~$\tau$ is multiplication by~$a^{-1}$.
The map on sets of minimal vectors from the fiber over~$1$ to the
fiber over~$\tau$ is $Q \mapsto M$, which is $Q \mapsto Qa$.
\end{proof}

\begin{corollary}
  For any two $\tau_1, \tau_2 \geqslant \tau_0$, the intersections of
  the well-tempered complex with the fibers over~$\tau_1$ and~$\tau_2$
  are equal to each other as subsets of $Y/K$, independently
  of~$\tau$;
\end{corollary}

\begin{proof}
  By Theorem~\ref{wtctau0}, both are equal to~$a^{-1}$ times the fiber
  over~$\tau = 1$.
\end{proof}



\section{Computing Hecke Operators}
\label{sec:hos}

Let the Hecke pair $(\Gamma_0, \Delta)$ be as above.  Let~$\rho$ be
any left $\Z\Delta$-module.  There is a natural left action of the
Hecke algebra on cohomology groups for~$\Gamma_0$ with coefficients
in~$\rho$ \cite[\S1.1]{AS}.  For $a\in\Delta$, the action of the Hecke
correspondence~$T_a$ on the cohomology is called the \emph{Hecke
  operator}, and it will also be denoted~$T_a$.  In this section, we
will give an algorithm that uses the well-tempered complex to
compute~$T_a$.

\subsection{Equivariant Cohomology}

First, we specify the cohomology theory we will use.  This is the
\emph{equivariant cohomology} $H^*_{\Gamma_0}(X; \rho)$, the
cohomology of the group~$\Gamma_0$ with coefficients in the complex of
cochains on~$X$ with coefficients in~$\rho$ \cite[VII.7]{Br}.

In this paragraph, say a prime is ``bad'' if it divides the order of a
finite subgroup of~$\Gamma_0$.  Since~$\Gamma_0$ is arithmetic, it has
at most finitely many finite subgroups up to conjugacy.  Hence there
are at most finitely many bad primes.  Suppose the module underlying
the coefficient system~$\rho$ is a vector space over a field whose
characteristic is not divisible by the bad primes; this includes
characteristic zero.  Then the equivariant cohomology
$H^*_{\Gamma_0}(X; \rho)$ is canonically isomorphic to the cohomology
$H^*(\Gamma_0\bs X; \rho)$ of the space $\Gamma_0\bs X$ with
coefficients in~$\rho$, as well as to the group cohomology
$H^*(\Gamma_0; \rho)$.

To compute equivariant cohomology, we may replace~$X$ with any acyclic
cell complex on which~$\Gamma_0$ acts.The symmetric space $X = G/K$ is
contractible.  By Theorem~\ref{wrrmain}, the well-rounded retract
$\widetilde{W}$ is a strong deformation retract of~$X$, hence
contractible, hence acyclic.  Thus $H^*_{\Gamma_0}(X; \rho) =
H^*_{\Gamma_0}(\widetilde{W}; \rho)$.  Again by Theorem~\ref{wrrmain},
the fiber $\widetilde{W}_\tau$ of the well-tempered complex
$\widetilde{W}^+$ over any~$\tau$ is a strong deformation retract
of~$X$, hence acyclic.  This holds in particular for the fibers
$\widetilde{W}_{\tau^{(i)}}$ over the critical temperaments $\tau^{(i)}$.
Denote the portion
of $\widetilde{W}^+$ with $u$ between two consecutive critical
temperaments, $u \in [u^{(i-1)}, u^{(i)}]$, by
$\widetilde{W}_{[u^{(i-1)}, u^{(i)}]}$.  By Theorem~\ref{wtccont},
$\widetilde{W}_{[u^{(i-1)}, u^{(i)}]}$ is a strong deformation retract
of $X \times [u^{(i-1)}, u^{(i)}]$, hence is acyclic.  Any of the
$\widetilde{W}_{u^{(i)}}$ or $\widetilde{W}_{[u^{(i-1)}, u^{(i)}]}$
can be used to compute the equivariant cohomology of~$\Gamma$.  We can
compute it in finite terms because the complexes have only finitely
many $\Gamma$-orbits of cells.  $\widetilde{W}_{[u^{(i-1)}, u^{(i)}]}$
has dimension one higher than the vcd, but its cohomology in degree
vcd$+1$ will be zero.

\subsection{Definition of the Hecke Operator}

Apply the equivariant cohomology functor to the
diagram~\eqref{twomaps}:
\begin{equation} \label{twomapsequivar}
\xymatrix{
  H^*_{\Gamma}(X; \rho) \ar@/_/[d]_{p_*} \\
  H^*_{\Gamma_0}(X; \rho) \ar@/_/[u]_{q^*}
}
\end{equation}
The Hecke operator~$T_a$ on $H^*_{\Gamma_0}(X; \rho)$ is by definition
\[
p_* q^* .
\]
The map $q^* : H^*_{\Gamma_0}(X; \rho) \to H^*_{\Gamma}(X; \rho)$ is
the natural pullback map for~$q$. The map $p_* : H^*_{\Gamma}(X; \rho)
\to H^*_{\Gamma_0}(X; \rho)$ is the transfer map \cite[III.9]{Br}
for~$p$, which is defined because $\Gamma = \Gamma_0 \cap a^{-1}
\Gamma_0 a$ has finite index in~$\Gamma_0$.

\subsection{Computing Hecke Operators}

We now describe how to use~\eqref{twomapsequivar} with the
well-tempered complex to compute the Hecke operators in practice.

First, we compute~$p_*$.  We use $\tau=1$, the first temperament, when
working with~$p$.  The retracts $\widetilde{W}$ and $\widetilde{W}_1$
are equal by definition. $\Gamma_0$ acts on $\widetilde{W}$, and the
smaller group $\Gamma$ acts on $\widetilde{W}_1$.  Computing the
transfer map is straightforward.  (In practice it is tricky to get the
orientation questions correct.  This is true for all the cells, and
especially for the cells with non-trivial stabilizer subgroups.  This
comment applies to all the computations in this paper.)

Next, we compute~$q^*$.  The pullback map is natural on cohomology,
but we must account for the factor of~$a$ in the definition of~$q$.
The key is to use the \emph{last} temperament~$\tau_0$ when working
with~$q$.  We compute $H^*_\Gamma(X; \rho)$ as
$H^*_\Gamma(\widetilde{W}_{\tau_0}; \rho)$.  By Theorem~\ref{wtctau0},
there is a homeomorphism of cell complexes $\widetilde{W}_{\tau_0} \to
\widetilde{W}_1$, from the last temperament to the first, given by
multiplication by~$a$.  As we saw for~$p$, $\widetilde{W}_1$
equals~$\widetilde{W}$.  Thus there is a cellular map which enables us
to compute $q^* : H^*_{\Gamma_0}(\widetilde{W}; \rho) \to
H^*_\Gamma(\widetilde{W}_{\tau_0}; \rho)$.

During the computation, the cells are represented by the combinatorial
data $(M, [u, u'])$ of Section~\ref{subsec:refining}.  By
Theorem~\ref{wtctau0}, the map in the last paragraph involves
\emph{dividing} $M$ by~$a$; the map sends $(M, [u, u'])$ to $(M
a^{-1}, [u, u'])$.  The Lemma before Theorem~\ref{wtctau0} showed
that, for the last temperament~$\tau_0$, all the index sets $M$ are
subsets of the sublattice $M_0 = L_0 a$.  This is why $M a^{-1}$ will
be a subset of~$L_0$ (will be ``integral'').

Computing only $p_*$ and $q^*$ does not give us the Hecke operator.
The map of Theorem~\ref{wtctau0} involves dividing or multiplying
by~$a$.  It is not a $\Gamma$-map, because $a\in\Delta$ but $a \notin
\Gamma$ in general.  For this reason, we cannot directly map
$H^*_\Gamma(\widetilde{W}_{\tau_0}; \rho)$ to
$H^*_\Gamma(\widetilde{W}_1; \rho)$.

To overcome this last difficulty, we use the whole well-tempered
complex to define a chain of morphisms and quasi-isomorphisms.  For
$i=1,\dots, i_r$, in the portion $\widetilde{W}_{[u^{(i-1)},u^{(i)}]}$
over the fibers $u \in [u^{(i-1)},u^{(i)}]$, define the closed
inclusions of the fibers on the left and right sides:
\[
  \widetilde{W}_{u^{(i-1)}}
  \xhookrightarrow{l^{(i-1)}}
  \widetilde{W}_{[u^{(i-1)},u^{(i)}]}
  \xhookleftarrow{r^{(i)}}
  \widetilde{W}_{u^{(i)}}
  \]
By Theorems~\ref{wtccont} and~\ref{wtccell}, we can compute the
pullbacks $(l^{(i-1)})^*$ and the pushforwards $(r^{(i)})_*$ on
$H^*_\Gamma(\dots; \rho)$.  The pullback is a naturally defined
cellular map.  The pushforward $(r^{(i)})_*$ is a quasi-isomorphism,
the inverse of the pullback $(r^{(i)})^*$; we compute the pullback at
the cochain level using the cellular map, then invert the map on
cohomology.

We summarize our algorithm as a theorem.

\begin{theorem} \label{wtcalgor}
  With notation as above, the Hecke operator $T_a$ on equivariant
  cohomology~\eqref{twomapsequivar} may be computed in finite terms as
  the composition
  \[
  p_* {l^{(0)}}^* r^{(1)}_* {l^{(1)}}^* r^{(2)}_* \cdots {l^{(i_r-1)}}^* r^{(i_r)}_* q^* .
  \]
\end{theorem}

Note that the maps $p_*$, $q^*$, and ${l^{(i)}}^*$ and $r^{(i)}_*$ for
the various~$i$ can be computed in parallel to speed up the
computation.

\subsection{Cohomology of Subgroups}
\label{subsec:levelN}

Let $\Gamma' \subseteq \Gamma_0$ be an arithmetic subgroup.  We wish
to compute Hecke operators on the equivariant cohomology
$H^*_{\Gamma'}(X; \rho)$ for any~$\Gamma'$.  By Shapiro's Lemma
\cite[III.6.2]{Br}, $H^*_{\Gamma'}(X; \rho) \cong H^*_{\Gamma_0}(X;
\mathrm{Coind}_{\Gamma'}^{\Gamma_0}\rho)$.  We use
Theorem~\ref{wtcalgor} to compute the latter.


\section{Computing the Well-Tempered Complex}
\label{sec:hecketope}

We now describe practical algorithms for finding a list of
representatives of the $\Gamma$-orbits of cells in the well-tempered
complex.

In broad terms, there are two kinds of steps.  First, we take an
individual~$\tau$ and find all the cells in $W_\tau$ up to
$\Gamma$-equivalence.  Second, we determine the interval of~$\tau'$
around the given~$\tau$ where the cell structure does not change.  The
endpoints of this interval are two consecutive members of the set of
critical temperaments.  We also carry out the first step at the
critical temperaments themselves.  We repeat the first and second
steps, starting at different~$\tau$, until the entire interval $[1,
  \tau_0]$ has been covered.

As we said in the Introduction, our algorithm for Hecke operators has
two parts, one-time work and every-time work.  The one-time work is to
compute~$\widetilde{W}^+$ for a given $L_0$, set of weights~$\varphi$,
and $a\in\Delta$.  The every-time work is to pick a~$\rho$ and compute
the Hecke operators on $H^*_{\Gamma_0}(X; \rho)$.  As in
Section~\ref{subsec:levelN}, computing Hecke operators for a range of
different levels~$N$ comes under the heading of every-time work.

\subsection{Finding the Cells for One Fiber}

In Section~\ref{wtccellpf} we introduced the space~$E_v$ of $n\times
n$ Hermitian matrices over~$D_v$.  Put the positive definite inner
product $\langle\langle A, B \rangle\rangle_v = \sum_{i,j=1}^n A_{ij}
B_{ij}^*$ on~$E_v$.  Then $\langle\langle A, B \rangle\rangle = \sum_v
\Tr^{D_v}_\R \langle\langle A, B \rangle\rangle_v$ is a positive
definite inner product on $E = \prod_v E_v$.  For a row vector $x\in
L_0$, let~$\psi = \psi(x)$ be the matrix whose component~$\psi_v$ at
the $v$-th place has $i,j$ entry equal to $(x_v)_i^* (x_v)_j$.  Let $Z
= g (g^*)^t$.  The next proposition generalizes Voronoi's original
work on perfect forms \cite[\S15]{Vor}.

\begin{proposition}
  With~$x$, $\psi$, and~$Z$ as above, $Z[x] = \langle\langle Z_v, \psi
  \rangle\rangle$.
\end{proposition}

\begin{proof}
For elements in the image of the embedding $D \hookrightarrow D_v$, we
drop the subscript~$v$ for simplicity.
\begin{align*}
  Z[x] &= \langle xg, xg \rangle \\
  &= \sum_v \sum_{i=1}^n \Tr^{D_v}_\R(((xg)_i)^* (xg)_i) \\
  &= \sum_v \Tr^{D_v}_\R \left( \sum_{i=1}^n ((xg)_i)^* (xg)_i \right) \\
  &= \sum_v \Tr^{D_v}_\R \left(
  \sum_{i=1}^n \left( \sum_{j=1}^n x_j g_{ji} \right)^*
  \left( \sum_{k=1}^n x_k g_{ki} \right)
  \right) \\
  &= \sum_v \Tr^{D_v}_\R \left(
  \sum_{i,j,k} g_{ji}^* x_j^* x_k g_{ki} 
  \right) \\
  &= \sum_v \Tr^{D_v}_\R \left(
  \sum_{i,j,k} g_{ki} g_{ji}^* x_j^* x_k
  \right) \\
  &= \sum_v \Tr^{D_v}_\R \left(
  \sum_{j,k} \left( \sum_i g_{ki} g_{ji}^* \right) x_j^* x_k
  \right) \\
  &= \sum_v \Tr^{D_v}_\R \left(
  \sum_{j,k} (Z_v)_{kj} (x_k^* x_j)^*
  \right) \\
  &= \sum_v \Tr^{D_v}_\R \langle\langle Z_v, \psi \rangle\rangle_v \\
  &= \langle\langle Z_v, \psi \rangle\rangle.
\end{align*}
\end{proof}

Now let the notation be as in Theorem~\ref{wtctau0}.
\[
\varphi_\tau(x) Z[x] = \left\{
\begin{array}{cl}
  \varphi(x) Z[x] & \quad\mathrm{if\ }x \in M_0 - \{0\}, \\
  \tau^2\varphi(x) Z[x] & \quad\mathrm{if\ }x \notin M_0.
\end{array} \right.
\]
Define
\[
\psi_\tau(x) = 
\left\{
\begin{array}{cl}
  \psi & \quad\mathrm{if\ }x \in M_0 - \{0\}, \\
  \tau^2 \psi & \quad\mathrm{if\ }x \notin M_0.
\end{array} \right.
\]
\begin{definition}
  For a given~$\tau \in [1, \tau_0]$, the \emph{Hecketope} at~$\tau$
  is the intersection of~$X$ with the convex hull~$\mathcal{H}_\tau$
  in~$E$ of the points $\psi_\tau(x)$ for all $x\in L_0 - \{0\}$.
\end{definition}
The \emph{Voronoi polyhedron} is the Hecketope at $\tau = 1$
\cite[\S A.3.4]{St}.

We must explain why we took the intersection with~$X$ in the
definition.  The points $\psi_\tau(x)$ lie on the boundary of the
cone~$X$ in~$E$.  They are not in~$X$ (assuming $n\geqslant 2$),
because their rank over~$D$ is one; they are positive semidefinite,
not positive definite.  If~$F$ is a face of~$\mathcal{H}_\tau$ and
$\mathrm{int}\,F$ is the relative interior of~$F$, then
either~$\mathrm{int}\,F$ lies in~$X$ or is disjoint from~$X$.  The
Hecketope contains only the $\mathrm{int}\,F$ that lie in~$X$.

We remark that $\mathcal{H}_\tau$ is not compact, because it has
infinitely many vertices $\psi_\tau$.  It is not locally finite at the
faces where $\mathrm{int}\,F$ does not lie in~$X$.

Examining Formulas~\eqref{wtcsigmatau2} gives the

\begin{corollary} There is a one-to-one, inclusion-reversing
  correspondence between the faces of the Hecketope at~$\tau$ and the
  faces of the fiber $\widetilde{W}_\tau$ of the well-tempered complex
  over~$\tau$.
\end{corollary}
For instance, a facet (codimension-1 face) of the Hecketope
corresponds to a vertex (dimension~0) of the well-rounded retract.
When $D=\Q$, the polyhedral cones in the Voronoi tiling of~$X$ as
in~\cite{Vor} are exactly the cones spanned by the faces of the
Voronoi polyhedron.

If~$F$ is a face of~$\mathcal{H}_\tau$ for which $\mathrm{int}\,F$ is
disjoint from~$X$, then Formulas~\eqref{wtcsigmatau2} determine a set
of points which are not well rounded.  That is why we exclude such
$\mathrm{int}\,F$ from the Hecketope.

\subsection{Finding Representative Cells modulo~$\Gamma$}

For a given~$\tau$, we need an algorithm to find a complete set of
representatives of the cells of $\widetilde{W}_\tau$ modulo~$\Gamma$.
By the last Corollary, it is equivalent to provide an algorithm to
find a complete set of representatives of the faces of the Hecketope
$\mathcal{H}_\tau \cap X$ modulo~$\Gamma$.  We describe two such
algorithms.

\subsubsection{Contiguous Forms}

The first algorithm generalizes Voronoi's own method.  Find a
facet~$F_0$ of the Hecketope at which to start.  Enumerate all the
facets $F_{01}, F_{02}, \dots$ of this facet (these have codimension
two in the Hecketope).  Each $F_{0j}$ is the border between exactly
two facets of the Hecketope; the first is $F_0$, and the second is a
new facet \emph{contiguous} to $F_0$.  Voronoi gave an algorithm for
finding the contiguous facet when $D=\Q$ \cite[\S\S22--23]{Vor}; a
modern formulation, with ideas for implementing it on computers,
appears in \cite[\S7.8]{Mart}.  Enumerate all the facets contiguous to
$F_0$ across $F_{01}, F_{02}, \dots$.  Continue this process
recursively.  When the algorithm computes a new facet, check whether
it is $\Gamma$-equivalent to one we have already found.  The algorithm
stops when all the new facets one finds are $\Gamma$-equivalent to old
facets.  Once we have a list of finitely many facets, it is
straightforward to enumerate their faces in all the smaller dimensions
and classify them up to $\Gamma$-equivalence.

Specialized, highly optimized algorithms have been designed for
particular~$D$ and~$n$.  Examples include \cite{perf8}
and~\cite{EvGS}.

\subsubsection{A Global Method}

At present, our code uses a second method.  We begin with an intuitive
description.  We choose a finite subset~$B \subset L_0 - \{0\}$,
roughly the lattice points in a ball around the origin.  We take the
convex hull~$\mathcal{H}'$ of $\{\psi_\tau(x) \mid x\in B\}$ in~$E$.
$\mathcal{H}'\cap X$ is a chunk of the Hecketope $\mathcal{H}\cap X$
near the origin of~$E$.  The chunk has been truncated far from the
origin (since $\mathcal{H}'$ is compact, while $\mathcal{H}$ is not).
Near the region where the truncation occurs, $\mathcal{H}'$ will have
faces which are not faces of $\mathcal{H}$.

Let us make precise how we choose~$B$.  Choose a reasonably good
lattice basis of~$L_0$ as $\Z$-lattice.  (If $D=\Q$ and $L_0 = \Z^n$,
choose the standard basis.)  Put on~$L_0$ the usual 2-norm for the
coordinates on that basis.  Choose a parameter~$c$, the \emph{vertex
  count}, which is the approximate number of points we want to have
in~$B$.  Using the classical formula for the volume of a ball in
$\dim_\R S^n$ dimensions, choose an~$r_1$ so that the ball of
radius~$r_1$ contains about $1.25c$ vectors of~$L_0 - \{0\}$.  (We
always count $x$ and $-x$ as a single vector.  Similarly, if $u\in D$
is a unit with norm of absolute value~1 in each $D_v$, we count $x$
and $xu$ as a single vector.  The extra 0.25 is a margin of error.)
Let $\mathcal{L}_1$ be the list of these $x\in L_0$ within
radius~$r_1$.  Choose~$r_2$ so that the ball of radius~$r_2$ contains
about $1.25c$ vectors of~$M_0 - \{0\}$.  Let $\mathcal{L}_2$ be the
list of these $x\in M_0$.  If $x \in \mathcal{L}_1 \cap
\mathcal{L}_2$, remove~$x$ from $\mathcal{L}_1$.  For each $x \in
\mathcal{L}_1 \cup \mathcal{L}_2$, compute $\psi_\tau(x)$.  Give
$\psi_\tau(x)$ a score; we have experimented with different scores,
but currently we use $\sum z^* z$ for~$z$ running down the diagonal of
the matrix $\psi_\tau(x)$.  Let~$B$ contain the best~$c$ vectors $x
\in \mathcal{L}_1 \cup \mathcal{L}_2$, where ``best'' means the scores
are the least.  If the scores are tied, take a few more vectors
than~$c$ until the tie is broken.  The main point of this algorithm is
that $\psi_\tau(x)$ is scaled by a factor of~$\tau^2$ when
$x\in\mathcal{L}_1$, but not when $x\in\mathcal{L}_2$.  When $\tau$ is
near~1, $B$ contains the shortest~$c$ vectors in $L_0 -\{0\}$.  But
when~$\tau$ is near~$\tau_0$, almost all the vectors in~$B$ are
in~$M_0 - \{0\}$.  For mid-range~$\tau$, $B$ contains a mixture.

Once~$B$ is chosen, we compute the convex hull $\mathcal{H}'$ of
$\{\psi_\tau(x) \mid x\in B\}$ in~$E$ using Sage's class
\texttt{Polyhedron} over~$\Q$ \cite{SageMath}.  The
\texttt{Polyhedron} class has a method which enumerates all the faces
of~$\mathcal{H}'$ in every dimension.  We keep those faces whose
relative interior lies in~$X$, discarding those where it does not.  We
use~\eqref{wtcsigmatau2} to check whether the face is really a face
of~$\mathcal{H}$, and we discard it if not.  The faces that remain
form a set $\check{W}_\tau$.  We find representatives of the
$\Gamma$-orbits in $\check{W}_\tau$ by a straightforward search.

If we choose~$c$ too small, we observe that the set $\check{W}_\tau$
will not be a closed subcomplex.  Some small-dimensional faces of
$\mathcal{H}$ will be computed incorrectly in $\mathcal{H}'$ near
where $\mathcal{H}$ is truncated.  Because of the inclusion-reversing
correspondence in the last Corollary, this means some
large-dimensional faces will be missing in $\check{W}_\tau$; for
example, two $k$-cells may be missing their common $(k-1)$-dimensional
face.  However, when~$c$ is large enough, we do get a closed cell
complex equal to $W_\tau$.

As we have said, computing the Hecketopes is one-time work.  We take
the position that one-time work can be relatively slow, as long as the
every-time work is fast.  Still, choosing a large~$c$ makes the
one-time work very slow, and it generates large intermediate files
until we have finished modding out by~$\Gamma$.  A plan for our future
work on the Hecketopes is to replace the global method with the method
of contiguous forms.

\subsection{Finding Consecutive Critical Temperaments}

In this section, let $F$ be a given facet of the Hecketope at~$\tau$.
Dually, $F$ is a vertex of the retract $\widetilde{W}_\tau$
over~$\tau$.  As in~\eqref{wtcsigmatau2}, $F$ is characterized by a
set~$M$ of minimal vectors in $L_0 - \{0\}$.  Because~$F$ is a vertex
of the retract, there is a unique~$Z$ satisfying~\eqref{wtcsigmatau2}
for~$M$ and~$\tau$.  As before, change variables to~$u$ by $u =
1/\tau^2$.

Our next goal is to find the largest interval $[u_1, u_2]$ containing
a given~$u$ so that, for all $u'\in (u_1, u_2)$, \eqref{wtcsigmatau2}
still has a unique solution at~$u'$ for our chosen~$F$ and~$M$.  At
the endpoints $u' = u_1$ or $u' = u_2$, the solution
to~\eqref{wtcsigmatau2} could be unique for a proper superset of~$M$
(in other words, some new independent minimal vectors could appear).


To find $[u_1, u_2]$, we observe that~\eqref{wtcsigmatau2}, for a
given~$M$, defines a pencil of Hermitian forms in the variable~$u'$.
The left-hand side of the equations is a linear function of the
entries of~$Z$.  The right-hand sides are either~$1$ or~$u'$,
depending on whether or not $x\in M_0$.  Using symbolic algebra in
Sage, we find the solution $Z = Z(u')$ of this pencil.  The entries of
the matrix $Z(u')$ are linear polynomials in~$u'$.

The coefficients in this pencil are in~$D$ (are ``rational''), since
all the $x\in L_0$.  Thus the symbolic computations we are about to do
are based on rational arithmetic.  For simplicity, in the rest of this
section, we describe our implementation of the algorithm for $D=\Q$,
where $S=\R$.  If~$D$ is a more general number field or division
algebra, the algorithm would need modifications, such as using complex
interval arithmetic when we use real interval arithmetic below.

It may be that $Z(u')$ only has a solution at the one point $u' = u$.
This happens if and only if we are at a critical temperament.  If it
happens, we stop and report that the largest interval $[u_1, u_2]$ is
$[u,u]$.  Otherwise, we continue searching for a larger interval.

The algorithm we are about to describe is like Voronoi's original
contiguous form algorithm \cite[\S\S22--23]{Vor} \cite[\S7.8]{Mart},
but adapted to the variation of the temperament.  We find $\det Z(u')$
symbolically as a polynomial~$f(u')$.  Replacing~$f$ by its gcd with
its derivative, we guarantee that~$f$ has no repeated roots.  Using
real interval arithmetic (Sage's \texttt{RIF}), we find approximations
to the real roots of~$f$.  Our~$u$ itself is not a root, since~$Z$ is
positive definite there.  The real roots closest to~$u$ on either side
are where~$Z$ become a degenerate Hermitian form.  The interval $[u_1,
  u_2]$ must be properly contained in the interval cut out by the real
roots, because of positive definiteness.  We now describe how to
find~$u_1$ (the left side); finding~$u_2$ is similar.  We find~$u_1$
using a bisection argument.  Let $u_r$ be the real root of~$f$ closest
to~$u$ on the left.  Start with $u'$ halfway between $u_r$ and $u$.
If the equations in~\eqref{wtcsigmatau2} have no new minimal vectors
beyond those in~$M$, move $u'$ further to the left, halfway between
its current position and $u_r$.  After enough bisections, we must find
some new minimal vectors, because we are getting closer and closer to
where $\det Z(u') = 0$, the boundary of~$X$.  Once we find a $u'$
where the set $M'$ of minimal vectors in the equations
in~\eqref{wtcsigmatau2} is a proper superset of~$M$, we are
essentially done.  The finite number of vectors in $M' - M$ give
finitely many conditions; the rightmost of those conditions
determines~$u_1$.

\subsection{Finding the Well-Tempered Complex}

We complete Section~\ref{sec:hecketope}, giving a practical algorithm
for finding a list of representatives of the $\Gamma$-orbits of cells
in the well-tempered complex~$\widetilde{W}^+$.  The input is~$L_0$, a
set of weights~$\varphi$, and $a\in\Delta$.  We begin by finding the
Hecketope at $u^{[0]} = 1$.  This finds the fiber of $\widetilde{W}^+$
over $u^{[0]}=1$ (the well-rounded retract of~\cite{Ash84}), together
with the smallest possible $u^{[1]}$ so that the cell structure does
not change for $u' \in \left(u^{[1]}, u^{[0]}\right]$.  At the
  critical temperament $u = u^{[1]}$, the cell structure has changed,
  and we compute the fiber over it.  Next, we take a $u' < u^{[1]}$
  but close to $u^{[1]}$, and we compute the fiber over~$u'$ together
  with the interval around~$u'$.  If $u'$ was close enough, then the
  right-hand endpoint of the interval will indeed be $u^{[1]}$ (if
  not, we pick $u'$ even closer to $u^{[1]}$ and start over).  The
  interval around $u'$ where the cell structure does not change is
  then $\left(u^{[2]}, u^{[1]}\right)$.  Here $u^{[2]}$ is the next
  critical temperament. We continue in this way.  We stop when the
  left-hand endpoint of the interval is $1/\tau_0^2$.  We have
  generated the points $u^{[0]}$, \dots, $u^{[i_r]}$ in right-to-left
  order, but we renumber them as~$u^{(i)}$ in left-to-right order to
  match~\eqref{defcrittemp}.


\section{Some Results for Subgroups of $\SL_3(\Z)$}
\label{sec:sl3ex}

Let $\Gamma_0(N)$ be the subgroup of $\SL_3(\Z)$ consisting of the
matrices congruent to $\left[\begin{smallmatrix} *&*&* \\ *&*&*
    \\ 0&0&* \end{smallmatrix}\right]$ modulo~$N$.  We present some
cohomology computations for $\Gamma_0(N)$ for levels $N=2$
through~$7$.  For these~$N$, the group $(\Z/N\Z)^\times$ of units
mod~$N$ is cyclic, and hence the group of Dirichlet characters
$(\Z/N\Z)^\times \to \C^\times$ is cyclic.  Let~$\chi_N$ be a
generator of this Dirichlet group.

\begin{table}
\begin{tabular}{|c|c||c|c|c|c|}
  \hline
  $N$ & $\eta$ & $H^0$ & $H^1$ & $H^2$ & $H^3$ \\
  \hline
  \hline
  $2$ & $1$ & $1\oplus\varepsilon\oplus\varepsilon^2$ & & & \\
  \hline
  \hline
  $3$ & $1$ & $1\oplus\varepsilon\oplus\varepsilon^2$ & & & \\
  \hline
  $3$ & $\pm1$ & & & $\chi_3\oplus\varepsilon\oplus\varepsilon^2$ &  \\
  & & & & $1\oplus\varepsilon\oplus\chi_3\varepsilon^2$ &  \\
  \hline
  \hline
  $4$ & $1$ & $1\oplus\varepsilon\oplus\varepsilon^2$ & &
                      & $1\oplus\varepsilon\oplus\varepsilon^2$\\
  \hline
  $4$ & $\pm1$ & & & $\chi_4\oplus\varepsilon\oplus\varepsilon^2$ &  \\
  & & & & $1\oplus\varepsilon\oplus\chi_4\varepsilon^2$ &  \\
  \hline
  \hline
  $5$ & $1$ & $1\oplus\varepsilon\oplus\varepsilon^2$ & & & \\
  \hline
  $5$ & $\pm1$ & & & & $1\oplus\chi_5^2\varepsilon\oplus\varepsilon^2$ \\
  \hline
  $5$ & $\chi_5$ & & & $\chi_5\oplus\varepsilon\oplus\varepsilon^2$ &  \\
  & & & & $1\oplus\varepsilon\oplus\chi_5\varepsilon^2$ &  \\
  \hline
  \hline
  $6$ & $1$ & $1\oplus\varepsilon\oplus\varepsilon^2$ & &
                      & $1\oplus\varepsilon\oplus\varepsilon^2 \quad(\dim 2)$ \\
  \hline
  $6$ & $\pm1$ & &
  & $\chi_6\oplus\varepsilon\oplus\varepsilon^2 \quad(\dim 2)$ & \\
  & &
  & & $1\oplus\varepsilon\oplus\chi_6\varepsilon^2 \quad(\dim 2)$ & \\
  \hline
  \hline
  $7$ & $1$ & $1\oplus\varepsilon\oplus\varepsilon^2$ & & & \\
  \hline
  $7$ & $\pm1$ & & & $\chi_7^3\oplus\varepsilon\oplus\varepsilon^2$ & \\
  & & & & $1\oplus\varepsilon\oplus\chi_7^3\varepsilon^2$ & \\
  & & & & $\varepsilon\oplus$ (7.3.b.a) & \\
  \hline
  $7$ & $\chi_7^2$ & & & & $1\oplus\chi_7^2\varepsilon\oplus\varepsilon^2$ \\
  \hline
  $7$ & $\chi_7$ & & & $\chi_7\oplus\varepsilon\oplus\varepsilon^2$ & \\
  & & & & $1\oplus\varepsilon\oplus\chi_7\varepsilon^2$ & \\
  \hline
\end{tabular}
\caption{Galois representations for the cohomology $H^i(\Gamma_0(N);
  \eta)$ of congruence subgroups of $\SL_3(\Z)$ with nebentype
  coefficients~$\eta$.}
\label{tab:sl3ex}
\end{table}

Any Dirichlet character~$\eta$ mod~$N$ defines a \emph{nebentype}
character $\Gamma_0(N) \to \C^\times$ by $A \mapsto \eta(a_{33})$,
where $a_{33}$ is the bottom right entry of~$A$ reduced mod~$N$.  We
denote the nebentype character by the same symbol~$\eta$.  In
Table~\ref{tab:sl3ex}, $\eta=1$ is the trivial character.  $\eta =
\pm1$ is the quadratic character, the unique character whose image is
$\{\pm1\}$.  Other~$\eta$ are given as powers of $\chi_N$.  When
nebentypes differ by an automorphism of the Dirichlet group, it
suffices to take one representative, since the automorphism induces an
isomorphism on cohomology.

We computed the cohomology groups $H^i(\Gamma_0(N); \eta)$ for all
$i=0,\dots,3$ and all nebentypes~$\eta$.  (The vcd is~$3$, so $H^i$
vanishes for $i>3$.)  We then computed the Hecke operators $T_{\ell,
  k}$ on these cohomology groups, for $\ell=2,3,5,7$ and $k=1,2,3$.
For simplicity, we omitted~$\ell$ if $\ell\mid N$.  We decomposed the
cohomology into its common Hecke eigenspaces for the operators that we
computed.

The results are in Table~\ref{tab:sl3ex}.  A blank cell means the
cohomology group is~$0$.  Each non-blank entry indicates a non-zero
Hecke eigenspace.  The eigenspaces have dimension~$1$, except for
those labeled $(\dim 2)$, which have dimension~$2$.

The computations were done in exact arithmetic over the smallest field
containing the image of~$\eta$.  This field was~$\Q$ when $\eta=1$
or~$\pm1$, and otherwise was the appropriate \texttt{QuadraticField}
from Sage.

If~$E$ is a Hecke eigenspace for the operators $T_{\ell,k}$ for
\emph{all} primes~$\ell\nmid N$ and $k=1,2,3$, then~$E$ has an
attached Galois representation.  Let $a_{\ell,k}$ be the eigenvalue
for $T_{\ell,k}$.  The \emph{Hecke polynomial} at~$\ell$ is
\[
1 - a_{\ell,1} X + \ell a_{\ell,2} X^2 - \ell^3 a_{\ell,3} X^3.
\]
This is the characteristic polynomial of the Galois representation
for~$E$ at the Frobenius over~$\ell$.  The factorization of the Hecke
polynomial determines the Galois representation.  For example,
let~$\varepsilon$ be the cyclotomic character, the Galois
representation taking the value~$\ell$ at~$\ell$ for all $\ell\nmid
N$.  For any Dirichlet character~$\chi$, a factor of $1 - \chi(\ell)
\ell^m X$ in the Hecke polynomials for all~$\ell$ indicates that the
Galois representation for~$E$ has a direct summand $\chi
\varepsilon^m$, the tensor product of~$\chi$ with~$m$ copies
of~$\varepsilon$.  As a second example coming from Eichler-Shimura, a
classical cuspidal Hecke eigenform of weight~$k$ and character~$\chi$
has Hecke polynomial $1 - a_\ell X + \ell \cdot \chi(\ell) \cdot
\ell^{k-2} X^2$, where $a_\ell$ is the cusp form's Hecke eigenvalue
at~$\ell$.  (For an introduction to these ideas, see \cite{AGG}
\cite{AGM1} and the references there.)

Our Hecke eigenspaces~$E$ are determined by only finitely many~$\ell$.
We say that a Galois representation \emph{seems to be attached} to~$E$
if our data matches the Galois representation for all the~$\ell$ we
have computed.  The entries in Table~\ref{tab:sl3ex} are the Galois
representations that seem to be attached to~$E$.  We found their
direct sum decompositions by factoring the Hecke polynomials.  We
believe we have computed enough~$\ell$ so that the Galois
representations really are attached; the results follow the patterns
observed for $\SL_3$ in \cite{AGG} \cite{AGMY} and for $\SL_4$ in
\cite{AGM1} \cite{AGM3} \cite{AGM7}.

(7.3.b.a) is the cusp form $q - 3q^{2} + 5q^{4} - 7q^{7} - 3q^{8} +
9q^{9} + O(q^{10})$.  This is the unique newform of level~7 and
weight~3.  Its name comes from the $L$-functions and Modular Forms
Database \cite{lmfdb}.


\bibliographystyle{plain} \bibliography{wtc}

\end{document}